\documentclass[11pt]{article}
\usepackage[utf8]{inputenc}
\usepackage{amsfonts}
\usepackage{amsmath}
\usepackage{graphicx}
\usepackage{amssymb}
\usepackage{graphicx}
\usepackage{amsthm}
\usepackage{t1enc}
\usepackage{subfig}
\usepackage{MnSymbol}
\usepackage{mathrsfs}
\usepackage{bbm}
\usepackage{mathtools}
\usepackage[shortlabels]{enumitem}
\usepackage{tikz}
\usetikzlibrary{matrix}
\usepackage{tikz-cd}
\usepackage[cal=boondoxo]{mathalfa}
\usetikzlibrary{angles,quotes}
\usepackage{url}
\usepackage{arydshln}
\usepackage{appendix}
\usepackage{authblk}

\newtheorem{theorem}{Theorem}[section]
\newtheorem{lemma}[theorem]{Lemma}
\newtheorem{proposition}[theorem]{Proposition}
\newtheorem{corollary}[theorem]{Corollary}
\newtheorem{definition}[theorem]{Definition}

\newtheorem{conjecture}[theorem]{Conjecture}

\theoremstyle{definition}
\newtheorem{remark}[theorem]{Remark}

\def\fp{\mathfrak{p}}
\def\cO{\mathcal{O}}
\def\ff{\mathfrak{f}}
\def\Z{\mathbb{Z}}
\def\ord{\textup{\textrm{ord}}}
\def\N{\textup{\textrm{N}}}
\def\rec{\textup{\textrm{rec}}}
\def\sign{\textup{\textrm{sign}}}
\def\Eis{\textup{\textrm{Eis}}}

\def\Supp{\textup{\textrm{Supp}}}
\def\Gal{\textup{\textrm{Gal}}}
\def\Hom{\textup{\textrm{Hom}}}
\def\Log{\textup{\textrm{Log}}}
\def\Meas{\textup{\textrm{Meas}}}
\def\Err{\textup{\textrm{Err}}}
\def\id{\textup{\textrm{id}}}

\def\Xint#1{\mathchoice
{\XXint\displaystyle\textstyle{#1}}{\XXint\textstyle\scriptstyle{#1}}%
{\XXint\scriptstyle\scriptscriptstyle{#1}}{\XXint\scriptscriptstyle\scriptscriptstyle{#1}}\!\int}

\def\XXint#1#2#3{{\setbox0=\hbox{$#1{#2#3}{\int}$}\vcenter{\hbox{$#2#3$}}\kern-.5\wd0}}

\def\multint{\Xint\times}

\usepackage[colorlinks,
pdfpagelabels,
pdfstartview = FitH,
bookmarksopen = true,
bookmarksnumbered = true,
linkcolor = black,
plainpages = false,
hypertexnames = false,
citecolor = black, urlcolor=black]{hyperref}

\tikzset{mynode/.style={font=\footnotesize,inner sep=0pt,text=black}
}

\title{On the Equality of Three Formulas for Brumer--Stark Units}
\author[1]{Samit Dasgupta}

\author[2]{Matthew H.\@ L.\@ Honnor }

\author[3]{Michael Spie\ss}

\affil[1]{Duke University, Durham, North Carolina, 27710, United States \newline Email address: samit.dasgupta@duke.edu}
\affil[2]{St Paul's School, Lonsdale Road,
London, SW13 9JT, United Kingdom}

\affil[3]{Faculty of Mathematics, Bielefeld University, D-33501 Bielfeld, Germany
\newline Email address: mspiess@math.uni-bielefeld.de}

\date{\today}

\begin{document}

\baselineskip 15pt
\maketitle

\begin{abstract}
    We prove the equality of three conjectural formulas for Brumer--Stark units. The first formula has essentially been proven, so the present paper also verifies the validity of the other two formulas. 
\end{abstract}

\tableofcontents

\section{Introduction}

In this paper we prove the equality of three conjectural formulas for Brumer--Stark units made by the first author in \cite{MR2420508}, and the first and third authors in \cite{MR3861805}, and \cite{MR3968788}.

One significance of this result is that the first formula has essentially been proven in \cite{intgrossstark}, so the present paper also verifies the validity of the other two formulas. Additionally, the third formula, made in \cite{MR3968788}, relates to a conjecture for the principal minors of the Gross--Regulator matrix. The validation of the third formula here gives a proof of this conjecture for the diagonal entries. Our work generalizes a partial result in this direction established in \cite{comparingformulas}.  See \cite{intgrossstark} for a discussion of the application of the these formulas toward explicit class field theory.

We now describe our results more precisely.  Let $F$ denote a totally real field, and let $H$ denote a finite abelian extension of $F$.  Write $G=\Gal(H/F)$. 
Let $R_\infty$ denote the set of archimedean places of $F$.
Let $R$ be a finite set of places of $F$ containing $R_\infty$ and the places that are ramified in $H$. Fix a prime ideal $\fp \not \in R$ that splits completely in $H$
and let $S = R \cup \{ \mathfrak{p} \}$.  Finally, we consider an auxiliary finite set $T$ of primes of $F$, disjoint from $S$ and satisfying a standard minor condition (see \S\ref{s:pzf}). From $\S3$ on, and for our main results, we make the simplifying assumption that $T$ consists of a single prime $\lambda$. The following conjecture was first stated by Tate and called the Brumer--Stark conjecture, \cite[Conjecture $5.4$]{MR656067}. 

\begin{conjecture}\label{c:bs}
Let $\mathfrak{P}$ be a prime in $H$ above $\mathfrak{p}$. There exists an element 
\begin{equation*}
u_T \in \mathcal{U}_{\mathfrak{p}} =\{ u \in H^\ast : \ \mid u \mid_v=1 \ \text{if } v \ \text{does not divide} \ \mathfrak{p} \}
\end{equation*} 
such that $u_T \equiv 1 \pmod{T}$, and for all $\sigma \in G$, we have 
$
    \ord_\mathfrak{P}(u_T^\sigma)= \zeta_{R,T}(H/F, \sigma,0). $
\end{conjecture}
Here $v$ ranges over all finite and archimedean places of $H$; in particular, each complex conjugation in $H$ acts as an inversion on $\mathcal{U}_{\mathfrak{p}}$.  The definition of the partial zeta function $ \zeta_{R,T}(H/F, \sigma,0)$ is recalled in \S\ref{s:pzf}.
The conjectural element $u_T \in \mathcal{U}_{\mathfrak{p}}$ satisfying Conjecture \ref{c:bs} is called the {\bf Brumer-Stark unit} for the data $(S,T,H,\mathfrak{P})$.

Conjecture~\ref{c:bs} has been recently proved by the first author and collaborators (see \cite{Brumerstark}, \cite{BrumerstarkoverZ}). It is convenient for us to package together $u_T$ and its conjugates over $F$ into an element of $H^\ast \otimes \Z[G]$ that we call the {\bf Brumer--Stark element}:
\[ u_\mathfrak{p} = \sum_{\sigma \in G} u_T^\sigma \otimes [\sigma^{-1} ] \in H^\ast \otimes \Z[G]. \]

There have been three formulas conjectured for the image of the Brumer--Stark element $u_\mathfrak{p}$ in $F_\mathfrak{p}^\ast \otimes \mathbb{Z}[G]$. 
In \cite{MR2420508} the first author conjectured a $p$-adic analytic formula for $u_\fp$ following the methods of Shintani and Cassou-Nogu\`es.  We denote this formula by $u_1$ and state it precisely in \S\ref{s:u1}. The other two formulas, which we denote $u_2$ and $u_3$, were defined in joint work of the first and third authors in \cite{MR3861805} and \cite{MR3968788}, respectively. Both of these formulas are cohomological in nature and are defined using the Eisenstein cocycle. They are stated precisely in \S\ref{s:u2}  and \S\ref{s:u3}. We remark that our definition for $u_3 \in F_\mathfrak{p}^* \otimes \Z[G]$ differs from that used in \cite{MR3968788} by a sign  (which acts as inversion on the left factor of the tensor product or negation on the right side of the tensor product), in order to state our results more cleanly.
The following combines the conjectures of the first author (for $i=1$) and the first and third authors (for $i=2,3$).

\begin{conjecture} \label{bsformconj}
For $i=1, 2, 3$ we have $u_i = u_\mathfrak{p}.$
\end{conjecture}

The main result of this paper is the following.

\begin{theorem} \label{papermainthm}
We have the following equalities between the three conjectural formulas for the Brumer--Stark element $u_\fp$, namely,
\[ u_1=u_2 = u_3  \text{ in } F_\fp^* \otimes \Z[G] . \]
\end{theorem}

Recent work of the first author with Kakde has proved that  $u_1 = u_\fp$ up to a root of unity under some mild assumptions. 
Write $\mu(F_\mathfrak{p}^\ast)$ for the group of roots of unity in $F_\mathfrak{p}^\ast$.

\begin{theorem}[Theorem 1.6, \cite{intgrossstark}] \label{thmforform} Suppose that
 the rational prime $p$ below $\mathfrak{p}$  is odd and unramified in $F$.  Suppose further that 
 there exists $\mathfrak{q} \in S$ that is unramified in $H$ whose associated Frobenius  $\sigma_\mathfrak{q}$
is the complex conjugation in $G$.  
Then Conjecture \ref{bsformconj} for $u_1$ holds up to multiplication by a root of unity in $F_\mathfrak{p}^\ast$:
\[ u_1 = u_\mathfrak{p} \ \text{in} \ (F_\mathfrak{p}^\ast / \mu(F_\mathfrak{p}^\ast)) \otimes \mathbb{Z}[G]. \]
\end{theorem}

\begin{remark}
Theorem~\ref{papermainthm} implies that $u_2=u_3 =  u_\mathfrak{p} \ \text{in} \ (F_\mathfrak{p}^\ast / \mu(F_\mathfrak{p}^\ast)) \otimes \mathbb{Z}[G]$ under the assumptions of Theorem \ref{thmforform}.
\end{remark}

In \S\ref{s:u23} we prove that $u_2=u_3$ via a direct cohomological calculation that was foretold in \cite{MR3968788}.
The proof that $u_1=u_3$, which takes up \S\ref{s:u12}, is more interesting and involves a new idea not present in prior work in this direction.  
It can be broken into two parts.
Suppose that $H/F$ is a CM abelian extension of conductor $\mathfrak{f}$ such that $\mathfrak{p}$ splits completely in $H$. We note that if $\mathfrak{q} \mid \ff$ then we must have $\mathfrak{q} \in R$. Denote by $E_+(\ff) \subset \cO_F^*$ the subgroup of totally positive units congruent to 1 modulo $\ff$. We then prove by a direct calculation that
\begin{equation} \label{e:fcong}
 u_1(\sigma) \equiv u_3(\sigma) \pmod{E_+(\mathfrak{f})},
 \end{equation}
where $u_i(\sigma)$ denotes the $\sigma$ component of $u_i$. 

Next, Let $\mathfrak{f}'$ be an auxiliary ideal of $\mathcal{O}_F$ that is divisible only by primes dividing $\mathfrak{f}$. Let $H' \supset H$ be another finite abelian CM extension of $F$ in which $\mathfrak{p}$ splits completely, such that the conductor of $H'/F$ divides $\mathfrak{f}\mathfrak{f}'$. In particular, the extension $H^\prime/F$ is unramified outside $R$. For each $\sigma \in G$, we then show the norm compatibility relation for $i=1,3$,
\begin{equation}     \label{normcompbasic1}
    u_i(\sigma, H) = \prod_{\substack{\tau \in \Gal(H^\prime / F) \\ \tau \mid_H = \sigma } } u_i(\tau, H^\prime).
\end{equation}
We remark here that showing the above equation first requires the proof that $u_2=u_3$. Applying (\ref{e:fcong}) with $H$ replaced by $H'$ and combining with (\ref{normcompbasic1}), we obtain 
\begin{equation} \label{e:ffcong}
 u_1(\sigma, H) \equiv u_3(\sigma, H) \pmod{E_+(\mathfrak{f} \ff')}. 
 \end{equation}
If $R \neq R_\infty$, then taking larger and larger conductors $\ff \ff'$ and passing to a limit, we obtain the desired result 
\[ u_1(\sigma, H) = u_3(\sigma, H) . \]
In the case $R=R_\infty$ we are required to do a little more work. The issue in this case is that $\mathfrak{f} = 1$ so there are no nontrivial ideals $\ff^\prime$ we may take.

In this case, by adding auxiliary primes into $R$, we are able to show that there exists $\varepsilon \in E_+$ such that for each $\sigma \in G$ we have
\[ u_1(\sigma, H) = \varepsilon u_3(\sigma, H) . \]
We then extend the definitions for $u_1$ and $u_2$ to work with the trivial extension $F/F$. We note that this was already done for $u_2$ in \cite{MR3861805}. In fact, $u_2$ is defined for any finite abelian extension $H/F$. Furthermore, in \cite[Proposition 6.3]{MR3861805} it is proved that $u_2(H/F)=1$ if $H$ has at least two real places. In particular, $u_2(F/F)=1$. We also prove that $u_1(F/F)=1$. By the norm compatibility property satisfied by $u_1$ and $u_3$ we have
    \[ 1= u_1(F) = \prod_{\sigma \in G} u_1(\sigma, H) = \varepsilon^{\mid G \mid} \prod_{\sigma \in G} u_3(\sigma, H) =  \varepsilon^{\mid G \mid} \prod_{\sigma \in G} u_2(\sigma, H) =  \varepsilon^{\mid G \mid} . \]
Thus $\varepsilon =1$. Therefore, $u_1 =  u_3$.

\bigskip
{\bf Acknowledgements}.  
The first author was supported by NSF grant DMS 1901939 for the duration of this project.
The second author is grateful for the support of the Heilbronn Institute for Mathematical Research and Imperial College London. The third author was supported by the Deutsche Forschungsgemeinschaft via the grant SFB-TRR 358/1 2023 — 491392403. We thank Mahesh Kakde for helpful discussions.

\section{Preliminaries for the multiplicative integral formula}

\subsection{Notation}\label{s:prelim:notation}

Recall that we have let $F$ be a totally real field of degree $n$ over $\mathbb{Q}$ with ring of integers $\mathcal{O} = \mathcal{O}_F$. Let $E_F=\mathcal{O}_F^\ast$ denote the group of global units. More generally, for a finite set $S$ of nonarchimedean places of $F$ we denote by $E_S = E_{F,S}$ the group of $S$-units of $F$. We define
\begin{equation}\label{e:sbar}
    \overline{S} = \{ \mathfrak{q} :  \mathfrak{q} \mid q  \ \text{where, for some} \ \mathfrak{r} \in S, \  \mathfrak{r} \mid q   \} .
\end{equation}
We also let $H/F$ be a totally complex abelian extension containing a CM-subfield. Let $\mathfrak{f}$ denote the conductor of the extension $H/F$. We write $E_+(\mathfrak{f})$ for the totally positive units of $F$ that are congruent to $1 \pmod{\mathfrak{f}}$. Write $G_\mathfrak{f}$ for the narrow ray class group of conductor $\mathfrak{f}$. Let $e$ be the order of $\mathfrak{p}$ in $G_\mathfrak{f}$ and suppose that $\mathfrak{p}^e = (\pi)$ with $\pi \equiv 1 \pmod{\mathfrak{f}}$ and $\pi$ totally positive. We write $\mathbb{O}= \mathcal{O}_\mathfrak{p} - \pi \mathcal{O}_\mathfrak{p} \subset F_\mathfrak{p}^\ast$.

Define $\mathbb{A}= \mathbb{A}_F$ as the adele ring of $F$. For a $\mathbb{Q}$-vector space $W$ fix the notation $W_{\widehat{\mathbb{Z}}}=W \otimes_\mathbb{Z} \widehat{\mathbb{Z}} = W \otimes_\mathbb{Q} \mathbb{A}_\mathbb{Q}^\infty$. Here $ \mathbb{A}_\mathbb{Q}^\infty$ denotes the finite adeles of $\mathbb{Q}$. For an abelian group $A$ and prime number $\ell$, we put $A_\ell = A \otimes_\mathbb{Z} \mathbb{Q}_\ell$. 

For a place $v$ of $F$ we put $U_v=\mathbb{R}_+=\{ x \in \mathbb{R} \mid x>0 \}$ if $v \mid \infty$ and $U_v=\mathcal{O}_v^\ast$ if $v$ is finite. For a set $S$ of places of $F$ we let $\mathbb{A}^S$ denote the adele ring away from $S$. We also define $U^S= \prod_{v \nin S} U_v$, and $U_S= \prod_{v \in S} U_v$. We shall also use the notations $F^S=( \mathbb{A}^S_F \times U_S) \cap F^\ast$ and  $F_+^S=( \mathbb{A}^S_F \times U_S) \cap F^\ast_+$.

Furthermore, for a nonarchimedean place $v$ of $F$ and an integer $m \geq 0$ we let $U_v^{(m)}$ denote the $m$-th higher units, i.e., $U_v^{(m)} \coloneqq \{ x \in U_v \mid \ord_v(x-1) \geq m \}$. If $\mathfrak{f}$ is an integral ideal and $S$ is a finite set of places we then put
\[ U_\mathfrak{f}^S \coloneqq \prod_{v \not \in S}U_v^{\ord_v(\mathfrak{f})} . \]

Finally we note that if we have a function $f :X \rightarrow Z$, with $X \subseteq Y$ and $Z$ an abelian group, then we can extend $f$ to a function $f_! : Y \rightarrow Z$ by defining 
\begin{equation}
    f_!(y)= \begin{cases}  f(y) &\text{if} \ y \in X  \\ 0 &\text{if} \ y \in Y-X ,  \end{cases}
\end{equation}
we call this the extension of $f$ to $Y$ by $0$.

\subsection{Partial zeta functions} \label{s:pzf}

For $\sigma \in G = \Gal(H/F)$, we define the \textbf{partial zeta function}
\begin{equation}     \label{pzf}
    \zeta_{R}(H/F, \sigma, s)= \sum_{\substack{(\mathfrak{a},R)=1 \\ \sigma_\mathfrak{a}=\sigma}} \N\mathfrak{a}^{-s}.
\end{equation}

Here the sum ranges over all integral ideals $\mathfrak{a}\subset \mathcal{O}$ that are relatively prime to the elements of $R$ and whose associated Frobenius element $\sigma_\mathfrak{a} \in G$ is equal to $\sigma$.
The series (\ref{pzf}) converges for $\text{Re}(s)>1$ and has a meromorphic continuation to $\mathbb{C}$, regular outside $s=1$. 
When the field extension $H/F$ is clear from context, we drop it from the notation and simply write $ \zeta_{R}(\sigma, s)$. Since $\fp$ splits completely in $H$, the zeta functions associated to the sets of primes $R$ and $S = R \cup \{\fp\}$ 
are related by the formula
\[ \zeta_{S}(\sigma, s)= (1-\N\mathfrak{p}^{-s}) \zeta_{R}(\sigma,s). \]

Recall that we have fixed an auxiliary finite set of primes of $F$, denoted $T$, that is disjoint from $S$.
The partial zeta function associated to the sets $R$ and $T$ is defined by the group ring equation
\begin{equation}     \label{pzfT}
    \sum_{\sigma \in G} \zeta_{R,T}(\sigma, s)[\sigma^{-1}]= \prod_{\eta \in T}(1-[\sigma_\eta^{-1}]\N\eta^{1-s}) \sum_{\sigma \in G} \zeta_{R}(\sigma, s)[\sigma^{-1}].
\end{equation}

We assume that the set $T$ contains at least two primes of different residue characteristic or at least one prime $\eta$ with absolute ramification degree at most $\ell-2$, where $\eta$ lies above $\ell$. With this in place, the values $\zeta_{R,T}(K/F, \sigma, 0)$ are rational integers for any finite abelian extension $K/F$ unramified outside $R$ and any $\sigma \in \Gal(K/F)$. This was shown by Deligne-Ribet \cite{MR579702} and Cassou-Nogu\'es \cite{MR524276}.

Our assumption on $T$ implies that there are no nontrivial roots of unity in $H$ that are congruent to $1$ modulo $T$. Thus the $\mathfrak{p}$-unit $u_T$ in Conjecture \ref{c:bs}, if it exists, is unique. Note also that our $u_T$ is actually the inverse of the $u$ in \cite[Conjecture $7.4$]{MR931448}.
From \S \ref{s:u1} onwards, for ease of notation, we fix $T=\{ \lambda \}$ for an appropriate choice of $\lambda$.

\subsection{Shintani zeta functions}

Shintani zeta functions are a crucial ingredient in each of the constructions we study. We establish the necessary notation here, following Shintani \cite{MR0427231}. In this subsection we continue to allow any choice of appropriate $T$.

For each $v \in R_\infty$ we write $\sigma_v : F \rightarrow \mathbb{R}$ and fix the order of these embeddings. We can then embed $F$ into $\mathbb{R}^n$ by $x \mapsto (\sigma_v(x))_{v\in R_\infty}$. Note that $F^\ast$ acts on $\mathbb{R}^n$ with $x \in F^\ast$ acting by multiplication by $\sigma_v(x)$ on the $v$-component of any vector in $\mathbb{R}^n$. For linearly independent $v_1, \dots ,v_r \in \mathbb{R}_+^n$, define the simplicial cone 
\[ C(v_1, \dots , v_r)= \left\{ \sum_{i=1}^r c_i v_i \in \mathbb{R}_+^n : c_i>0 \right\}. \]

\begin{definition}
A \textbf{Shintani cone} is a simplicial cone $C(v_1, \dots , v_r)$ generated by elements $v_i \in F \cap \mathbb{R}_+^n$. A \textbf{Shintani set} is a subset of $\mathbb{R}_+^n$ that can be written as a finite disjoint union of Shintani cones. 
\end{definition}

We now recall the definition of \textbf{Shintani zeta functions}. Let $\mathcal{O}_{F,\mathfrak{p}}$ denote the ring of $\mathfrak{p}$-integers of $F$. For any fractional ideal $\mathfrak{b} \subset F$ relatively prime to $S$, we let $\mathfrak{b}_{\mathfrak{p}} = \mathfrak{b} \otimes_{\mathcal{O}_F} \mathcal{O}_{F,\mathfrak{p}}$ denote the $\mathcal{O}_{F,\mathfrak{p}}$-module generated by $\mathfrak{b}$. Write $\mathfrak{f}$ for the conductor of the extension $H/F$. Let $\mathfrak{b}$ be a fractional ideal of $F$ relatively prime to $S= \{ \mathfrak{p} \} \cup R $ and $\overline{T}$, and let ${D}$ be a Shintani set. For each compact open $U \subseteq F_\mathfrak{p}$, define, for $\text{Re}(s)>1$,
\begin{equation*}     \label{eqn20}
    \zeta_{R}(\mathfrak{b}, {D}, U,s)=\N \mathfrak{b}^{-s} \sum_{\substack{\alpha \in F \cap {D}, \  \alpha \in U \\ (\alpha, R)=1, \ \alpha \in \mathfrak{b}_\mathfrak{p}^{-1} \\ \alpha \equiv 1 \pmod{\mathfrak{f}} }} \N \alpha^{-s} .
\end{equation*}
For a general element $z \in F^\ast$ the congruence $z \equiv 1 \pmod{\mathfrak{f}}$ means that $z-1 \in \mathfrak{f} \mathcal{O}_\mathfrak{f} \cap F $, where $\mathcal{O}_\mathfrak{f}$ is the $\mathfrak{f}$-adic completion of $\mathcal{O}_F$. We define $\zeta_{R,T}(\mathfrak{b}, {D},U,s)$ in analogy with (\ref{pzfT}). Suppose that 
\[ \prod_{\eta \in T}(1-[\eta]\N\eta^{1-s}) = \sum_{\mathfrak{a}} c_\mathfrak{a}(s)[\mathfrak{a}] \]
in the group ring of fractional ideals with coefficients in the ring of complex valued functions on $\mathbb{C}$, and define
\begin{equation}     \label{shintanizetafn}
    \zeta_{R,T}(\mathfrak{b}, D, U, s)=  \sum_{\mathfrak{a}} c_\mathfrak{a}(s) \zeta_{R}(\mathfrak{a}^{-1}\mathfrak{b}, D, U, s).
\end{equation}
In particular, if $\N \eta = \ell$ and $T=\{ \eta \}$, we have
\[ \zeta_{R,T}(\mathfrak{b}, D, U, s) = \zeta_{R}(\mathfrak{b}, D, U, s) - \ell^{1-s} \zeta_{R}(\mathfrak{b}\eta^{-1}, D, U, s) . \]
It follows from Shintani's work in \cite{MR0427231} that the function $\zeta_{R,T}(\mathfrak{b}, {D},U,s)$ has a meromorphic continuation to $\mathbb{C}$. We now want to define conditions on the set of primes $T$ and the Shintani set $D$ to allow our Shintani zeta functions to be integral at $0$.

\begin{definition}
A prime ideal $\eta$ of $F$ is called \textbf{good} for a Shintani cone $C$ if
\begin{itemize}
    \item $ \N \eta$ is a rational prime $\ell$; and 
    \item the cone $C$ may be written $C=C(v_1, \dots , v_r)$ with $v_i \in \mathcal{O}$ and $v_i \nin \eta$.
\end{itemize}
We also say that $\eta$ is \textbf{good} for a Shintani set $D$ if $D$ can be written as a finite disjoint union of Shintani cones for which $\eta$ is good.
\end{definition}

\begin{definition}
The set $T$ is \textbf{good} for a Shintani set ${D}$ if ${D}$ can be written as a finite disjoint union of Shintani cones $D=\bigsqcup C_i$ so that for each cone $C_i$, there are at least two primes in $T$ that are good for $C_i$ (necessarily of different residue characteristic by our earlier assumption) or one prime $\eta \in T$ that is good for $C_i$ such that $\N\eta \geq n+2$.
\end{definition}

\begin{remark}
Given any Shintani set $D$, it is possible to choose a set of primes $T$ such that $T$ is good for $D$. In fact, all but a finite number of prime ideals with prime norm are good for a given Shintani set.
\end{remark}

We can now note the required property to allow our Shintani zeta functions to be integral at zero. The proposition below is proved in \cite[p.15]{MR2420508}.

\begin{proposition}
If the set of primes $T$ is good for a Shintani set $D$, then
\[ \zeta_{R,T}(\mathfrak{b}, D,U,0) \in \mathbb{Z}. \]
\end{proposition}

We define a $\mathbb{Z}$-valued measure $\nu_T(\mathfrak{b}, {D})$ on $\mathcal{O}_\mathfrak{p}$ by
\begin{equation}     \label{eqn21}
    \nu_T(\mathfrak{b}, {D},U) \coloneqq \zeta_{R,T}(\mathfrak{b}, {D},U,0),
\end{equation}
for $U \subseteq \mathcal{O}_\mathfrak{p}$ compact open.

We are mostly interested in a particular type of Shintani set, one which is a fundamental domain for the action of a finite index subgroup $V \subset E_+(\mathfrak{f})$.

\begin{definition}
Let $V \subset E_+(\mathfrak{f})$ be a finite index subgroup (which is necessarily free of rank $n-1$). We call a Shintani set ${D}$ a \textbf{Shintani domain} for $V$ if ${D}$ is a fundamental domain for the action of $V$
on $\mathbb{R}_+^n$. That is,
\[ \mathbb{R}_+^n=\bigsqcup_{\epsilon \in V} \epsilon {D} \quad \text{(disjoint union).} \]
\end{definition}

The existence of such domains follows the work of Shintani, in particular from \cite[Proposition $4$]{MR0427231}. We note here some simple equalities that follow from the definitions. More details are given in \S 3.3 of \cite{MR2420508}. Recall we have written $G_\mathfrak{f}$ for the narrow ray class group of conductor $\mathfrak{f}$. Let $e$ be the order of $\mathfrak{p}$ in $G_\mathfrak{f}$, and write $\mathfrak{p}^e = (\pi)$ with $\pi \equiv 1 \pmod{\mathfrak{f}}$ and $\pi$ totally positive. We denote by $H_\mathfrak{f}$ the narrow ray class field of $F$ of conductor $\mathfrak{f}$, and by $H$ the maximal subfield of $H_{\mathfrak{f}}$ containing $F$ in which the prime $\mathfrak{p}$ splits completely. Let ${D}$ be a Shintani domain for $E_+(\mathfrak{f})$ and write $\mathbb{O}= \mathcal{O}_\mathfrak{p}-\pi \mathcal{O}_\mathfrak{p} $. Then,
\[
    \nu_T(\mathfrak{b}, {D},\mathbb{O}) = \zeta_{S,T}(H/F, \mathfrak{b}, 0) = 0, \quad \text{and} \quad
    \nu_T(\mathfrak{b}, {D},\mathcal{O}_\mathfrak{p}) = \zeta_{R,T}(H_\mathfrak{f}/F, \mathfrak{b}, 0).
\]

We now give two technical definitions that are necessary in the definition of $u_1$.

\begin{proposition}
Let $V \subset E_+(\mathfrak{f})$ be a finite index subgroup. Let ${D}$ and ${D}^\prime$ be Shintani domains for $V$. We may write ${D}$ and ${D}^\prime$ as finite disjoint unions of the same number of simplicial cones 
\begin{equation}
    {D}= \bigcup_{i=1}^d C_i, \quad {D}^\prime = \bigcup_{i=1}^d C_i^\prime,
    \label{eqn9}
\end{equation}
with $C_i^\prime = \epsilon_i C_i$ for some $\epsilon_i \in V$, $i=1, \dots , d$.
\end{proposition}

\begin{proof}
\cite[Proposition 3.15]{MR2420508} proves this result when $V=E_+(\mathfrak{f})$. The proof of this proposition is analogous.
\end{proof}

A decomposition as in (\ref{eqn9}) is called a \textbf{simultaneous decomposition} of the Shintani domains $({D}, {D}^\prime)$. 

\begin{definition}
Let $(D,D^\prime)$ be a pair of Shintani domains. A set $T$ is \textbf{good} for the pair $({D}, {D}^\prime)$ if there is a simultaneous decomposition as in (\ref{eqn9}) such that for each cone $C_i$, there are at least two primes in $T$ that are good for $C_i$, or there is one prime $\eta \in T$ that is good for $C_i$ such that $\N \eta \geq n+2$.
\end{definition}

\begin{definition}
Let $D$ be a Shintani domain. If $\beta \in F^\ast$ is totally positive, then $T$ is $\beta$-\textbf{good} for ${D}$ if $T$ is good for the pair $({D}, \beta^{-1} {D})$.
\end{definition}

\begin{lemma}[Lemma 3.20, \cite{MR2420508}]  \label{changeofvariable}
 Let $D$ be a Shintani set and $U$ a compact open subset of $\mathcal{O}_\mathfrak{p}$. Let $\mathfrak{b}$ be a fractional ideal of $F$, and let $\beta \in F^\ast$ be totally positive so that $\beta \equiv 1 \pmod{\mathfrak{f}}$ and $\ord_\mathfrak{p}(\beta) \geq 0$. Suppose that $\mathfrak{b}$ and $\beta$ are relatively prime to $R$ and that $\mathfrak{b}$ is also relatively prime to  $\overline{T}$. Let $\mathfrak{q}= (\beta)\mathfrak{p}^{-\ord_\mathfrak{p}(\beta)}$. Then
 \[ \zeta_{R,T}(\mathfrak{bq}, {D},U,0) = \zeta_{R,T}(\mathfrak{b}, \beta {D},\beta U,0). \]
\end{lemma}

We end this section with a lemma of Colmez that allows us to give an explicit Shintani domain. Let $\alpha$ be, up to a sign, one of the standard basis vectors of $\mathbb{R}^n$. Note that its ray ($\alpha\mathbb{R}_+$) is preserved by the action of $\mathbb{R}_+^n$. We define $\overline{C}_\alpha(v_1, \dots , v_r)$ to be the union of the cone $C(v_1, \dots , v_r)$ with the boundary cones that are brought into the interior of the cone by a small perturbation by $\alpha$, i.e., the set whose characteristic function is given by 
\begin{equation}
    \mathbbm{1}_{\overline{C}_\alpha(v_1, \dots , v_r)}(x) = \lim_{h \rightarrow 0^+} \mathbbm{1}_{C(v_1, \dots, v_r)}(x+h\alpha). 
    \label{overlinec}
\end{equation}
We use the usual bar notation for homogeneous chains
\[ [x_1 \mid \dots \mid x_{n-1}] = (1, x_1, x_1 x_2, \dots , x_1 \dots x_{n-1}). \]
Let $x_1, \dots , x_{n-1} \in F$. We define the sign map $\delta : F^{n} \rightarrow \{ -1,0,1 \} $ by the rule 
\begin{equation}
    \delta(x_1, \dots , x_n)=  \sign(\det  (\omega(x_1, \dots , x_n))),
    \label{epsilonnotation}
\end{equation}
where $\omega(x_1, \dots , x_n)$ denotes the $n \times n$ matrix whose columns are the images of the $x_i$ in $\mathbb{R}^n$. We adopt the convention $\sign(0)=0$.

\begin{lemma}[Lemma 2.2, \cite{MR922806}] \label{colmezlemma}
Let $\alpha$ be, up to a sign, one of the standard basis vectors of $\mathbb{R}^n$. Let $\varepsilon_1, \dots, \varepsilon_{n-1} \in E_+(\mathfrak{f})$ such that $V=\langle \varepsilon_1, \dots, \varepsilon_{n-1}  \rangle \subset E_+(\mathfrak{f}) $ has finite index. Suppose that for all $\tau \in S_{n-1}$ we have
\[ \delta([\varepsilon_{\tau(1)} \mid \dots \mid \varepsilon_{\tau(n-1)}])=\sign(\tau). \]
Then the Shintani set
\[ {D}= \bigcup_{\tau \in S_{n-1}}\overline{C}_{\alpha}([\varepsilon_{\tau(1)} \mid \dots \mid \varepsilon_{\tau(n-1)}]),  \]
is a Shintani domain for $V$.
\end{lemma}

For more details on the above lemma we refer to \cite[\S 1.3]{MR3351752}. We note also that another proof of the lemma is given in \cite[Corollary 2]{MR3198753}.

The existence of Shintani domains follows from the work of Shintani in \cite{MR0427231}. In Lemma \ref{lemmatomakembig} we show the existence of units $\varepsilon_1, \dots, \varepsilon_{n-1} \in E_+(\mathfrak{f})$ that satisfy the conditions of Lemma \ref{colmezlemma}.

\section{The multiplicative integral formula ($u_1$)} \label{s:u1}

\begin{definition} \label{multintdefinition}
Let $I$ be an abelian topological group that may be written as an inverse limit of discrete groups
\[ I= \varprojlim I_\alpha . \]
Denote the group operation on $I$ multiplicatively. For each $i \in I_\alpha$, denote by $U_i$ the open subset of $I$ consisting of the elements that map to $i$ in $I_\alpha$. Suppose that $G$ is a compact open subset of a quotient of $\mathbb{A}_F^\ast$  . Let $f:G \rightarrow I$ be a continuous map, and let $\mu$ be a $\mathbb{Z}$-valued measure on $G$. We define the \textbf{multiplicative integral}, written with a cross through the integration sign, by 
\begin{equation*}
    \multint_G f(x) d \mu (x)= \varprojlim \prod_{i \in I_\alpha}i^{\mu (f^{-1}(U_i))} \in I.
\end{equation*}
\end{definition}

Let $\lambda$ be a prime of $F$ such that $\N \lambda =\ell$ for a prime number $\ell\in \mathbb{Z}$ and $\ell \geq n+2$. We assume that no primes in $S$ have residue characteristic equal to $\ell$. In this section and from this point on we take $T = \{ \lambda \}$.

\begin{definition}
Let $\mathcal{D}$ be a Shintani domain for $E_+(\mathfrak{f})$, and assume that $\lambda$ is $\pi$-good for $\mathcal{D}$. Define the \textbf{error term} 
\begin{equation}
    \epsilon(\mathfrak{b}, \mathcal{D}, \pi ) \coloneqq \prod_{\epsilon \in E_+(\mathfrak{f})} \epsilon^{\nu_\lambda(\mathfrak{b}, \epsilon \mathcal{D} \cap \pi^{-1} \mathcal{D}, \mathcal{O}_\mathfrak{p})} \in E_+(\mathfrak{f}) .
    \label{errorterm}
\end{equation}
\end{definition}

By \cite[Lemma $3.14$]{MR2420508}, only finitely many of the exponents in (\ref{errorterm}) are nonzero. \cite[Proposition $3.12$]{MR2420508} and the assumption that $\lambda$ is $\pi$-good for $\mathcal{D}$ implies that the exponents are integers. We recall from (\ref{eqn21}) that the measure is defined as
\[ \nu_\lambda(\mathfrak{b}, \epsilon \mathcal{D} \cap \pi^{-1} \mathcal{D}, \mathcal{O}_\mathfrak{p}) = \zeta_{R, \lambda} (\mathfrak{b}, \epsilon \mathcal{D} \cap \pi^{-1} \mathcal{D}, \mathcal{O}_\mathfrak{p},0) . \]
We are now ready to write down the conjectural formula from \cite{MR2420508}. We note that for any Shintani domain $\mathcal{D}$ we can always choose a prime $\lambda$ that is $\pi$-good for $\mathcal{D}$. In fact, all but a finite number of primes will satisfy this property. Henceforth, we can assume that $\lambda$ satisfies the property written above and is $\pi$-good for $\mathcal{D}$. We now give the main definition of this section.
\begin{definition} \label{defnformula}
Let $\mathcal{D}$ be a Shintani domain for $E_+(\mathfrak{f})$, and assume that $\lambda$ is $\pi$-good for $\mathcal{D}$. Define 
\begin{equation*}
    u_{\mathfrak{p},\lambda}(\mathfrak{b}, \mathcal{D}) \coloneqq \epsilon(\mathfrak{b}, \mathcal{D}, \pi ) \pi^{\zeta_{R,\lambda}(H_\mathfrak{f}/F, \mathfrak{b},0)} \multint_\mathbb{O} x \ d \nu_\lambda (\mathfrak{b}, \mathcal{D},x) \in F^\ast_\mathfrak{p}.
\end{equation*}
\end{definition}

As our notation suggests, we have the following proposition.

\begin{proposition}[Proposition 3.19, \cite{MR2420508}] \label{prop3.19}
The element $u_{\mathfrak{p},\lambda}(\mathfrak{b}, \mathcal{D})$ does not depend on the choice of generator $\pi$ of $\mathfrak{p}^e$.
\end{proposition}

The following is conjectured.

\begin{conjecture}[Conjecture 3.21, \cite{MR2420508}] \label{conj3.21}
Let $e$ be the order of $\mathfrak{p}$ in $G_\mathfrak{f}$, and suppose that $\mathfrak{p}^e=(\pi)$ with $\pi$ totally positive and $\pi \equiv 1 \pmod{\mathfrak{f}}$. Let $\mathcal{D}$ be a Shintani domain for $E_+(\mathfrak{f})$, and let $\lambda$ be $\pi$-good for $\mathcal{D}$. Let $\mathfrak{b}$ be a fractional ideal of $F$ relatively prime to $S$ and $\ell$. We have the following.
\begin{enumerate}
    \item The element $u_{\mathfrak{p},\lambda}(\mathfrak{b}, \mathcal{D}) \in F_\mathfrak{p}^\ast$ depends only on the class of $\mathfrak{b} \in G_\mathfrak{f}/ \langle \mathfrak{p} \rangle$ and no other choices, including the choice of $\mathcal{D}$, and hence may be denoted $u_{\mathfrak{p},\lambda}(\sigma_\mathfrak{b})$, where $\sigma_\mathfrak{b} \in \Gal(H/F)$.
    
    \item The element $u_{\mathfrak{p},\lambda}(\sigma_\mathfrak{b})$ lies in $\mathcal{U}_\mathfrak{p}$, and $u_{\mathfrak{p},\lambda}(\sigma_\mathfrak{b}) \equiv 1 \pmod{\lambda}$. 
    
    \item Shimura reciprocity law: For any fractional ideal $\mathfrak{a}$ of $F$ prime to $S$ and to $\ell$, we have 
    \begin{equation*}
        u_{\mathfrak{p},\lambda}(\sigma_\mathfrak{ab})=u_{\mathfrak{p},\lambda}(\sigma_\mathfrak{b})^{\sigma_\mathfrak{a}}.
    \end{equation*}
\end{enumerate}
\end{conjecture}

As we noted in the introduction, this conjecture has been proved up to a root of unity (Theorem \ref{thmforform}).
We want to state the formula over $F_\mathfrak{p}^\ast \otimes  \mathbb{Z}[G]$ to match with the cohomological constructions.
\begin{definition}
We define
\[ u_1 = \sum_{\mathfrak{b} \in G_\mathfrak{f}/ \langle \mathfrak{p} \rangle } u_{\mathfrak{p},\lambda}(\mathfrak{b}, \mathcal{D}) \otimes [\sigma_\mathfrak{b}^{-1}] \in F_\mathfrak{p}^\ast \otimes \mathbb{Z}[G] .  \]
\end{definition}

\subsection{Transferring to a subgroup}

In this section we recall the results \cite{comparingformulas}, which allow us to transfer to a subgroup. Let $V$ be a finite index subgroup of $E_+(\mathfrak{f})$. Recall that $\pi$ is totally positive, congruent to $1$ modulo $\mathfrak{f}$ and satisfies $(\pi)= \mathfrak{p}^e$ where $e$ is the order of $\mathfrak{p}$ in $G_\mathfrak{f}$. Let $\mathcal{D}_V^\prime$ be a Shintani set which is a fundamental domain for the action of $V$ on $\mathbb{R}_+^n$ and assume that $\lambda$ is $\pi$-good for $\mathcal{D}_V^\prime$. As before, we shall refer to such Shintani sets as Shintani domains for $V$. Let $\mathfrak{b}$ be a fractional ideal of $F$ relatively prime to $S$ and $\ell$.

We define
\[ u_1(V, \sigma_\mathfrak{b}) =  u_{\mathfrak{p},\lambda}(\mathfrak{b}, \mathcal{D}_V^\prime) \coloneqq \prod_{\epsilon \in V} \epsilon^{\zeta_{R,\lambda}(\mathfrak{b}, \epsilon \mathcal{D}_V^\prime \cap \pi^{-1} \mathcal{D}_V^\prime, \mathcal{O}_\mathfrak{p},0) } \pi^{\zeta_{R,\lambda}( \mathfrak{b}, \mathcal{D}_V^\prime,\mathcal{O}_\mathfrak{p},0)} \multint_\mathbb{O} x \ d \nu_\lambda (\mathfrak{b}, \mathcal{D}_V^\prime,x), \]
and write $ u_1(V)=\sum_{\sigma \in G} u_1(V,\sigma) \otimes [\sigma^{-1}] $.

\begin{proposition}[Proposition 6.11, \cite{comparingformulas}] \label{canchangedom}
Let $\mathcal{K}$ and $\mathcal{K}^\prime$ be two Shintani domains for $V$ and $\lambda$ a prime of $F$ such that $\lambda$ is $\pi$-good for $\mathcal{K}$ and $\mathcal{K}^\prime$. If $\lambda$ is also good for $(\mathcal{K}, \mathcal{K}^\prime)$, then $u_{\mathfrak{p}, \lambda}(\mathfrak{b},  \mathcal{K}) = u_{\mathfrak{p}, \lambda}(\mathfrak{b},  \mathcal{K}^\prime)$.
\end{proposition}

Let $V \subset E_+(\mathfrak{f})$ be a finite index subgroup. The following proposition shows the relation between $u_1(\sigma)$ and $u_1(V, \sigma)$.

\begin{proposition}[Proposition 6.12, \cite{comparingformulas}] \label{firstchangeofdom}
Let $\mathcal{D}$ be a Shintani domain for $E_+(\mathfrak{f})$. Let $V$ be a finite index subgroup of $E_+(\mathfrak{f})$. Write $g_1, \dots , g_{n-1}$ for a $\mathbb{Z}$-basis of $E_+(\mathfrak{f})$ such that $g_1^{b_1}, \dots , g_{n-1}^{b_{n-1}}$ is a $\mathbb{Z}$-basis for $V$. Define 
\[ \mathcal{D}_V \coloneqq \bigcup_{j_1=0}^{b_1-1} \dots  \bigcup_{j_{n-1}=0}^{b_{n-1}-1} g_1^{j_1} \dots g_{n-1}^{j_{n-1}} \mathcal{D}. \]
Then, if $b_1, \dots , b_{n-1} > M$, where $M = M(\pi, g_1, \dots , g_{n-1})$ is some constant that depends on $g_1, \dots , \\ g_{n-1}$ and $\pi$ (up to multiplication by an element of $E_+(\mathfrak{f})$), we have
\[ u_{\mathfrak{p}, \lambda}(\mathfrak{b},  \mathcal{D}_V) = u_{\mathfrak{p}, \lambda}(\mathfrak{b},  \mathcal{D})^{[E_+(\mathfrak{f}):V]}. \]
\end{proposition}

\section{Preliminaries for the cohomological formulas} \label{s:precohom}

\subsection{Continuous maps} \label{s:cont}

For topological spaces $X$ and $Y$ let $C(X,Y)$ denote the set of continuous maps $X \rightarrow Y$. If $R$ is a topological ring we let $C_c(X,R)$ denote the subset of $C(X,R)$ of continuous maps with compact support. If we consider $Y$ (resp.\ $R$) with the discrete topology then we shall also write $C^0(X,Y)$ (resp.\ $C_c^0(X,R)$) instead of $C(X,Y)$ (resp.\ $C_c(X,R)$).

Assume now that $X$ is a totally disconnected topological Hausdorff space and $A$ a locally profinite group. We define subgroups $C^\diamond(X,A) \subseteq C(X,A)$ and $C_c^\diamond(X,A) \subseteq C_c(X,A)$ by
\[ C^\diamond(X,A)=C^0(X,A) + \sum_K C(X,K), \]
\[ C_c^\diamond(X,A)=C_c^0(X,A) + \sum_K C_c(X,K), \]
where the sums are taken over all compact open subgroups $K$ of $A$. So $C_c^\diamond(X,A)$ is the subgroup of $C_c(X,A)$ generated by locally constant maps with compact support $X \rightarrow A$ and by continuous maps with compact support $X \rightarrow K \subseteq A$ for some compact open subgroup $K \subseteq A$. Similarly $C^\diamond(X,A)$ is the subgroup of $C(X,A)$ generated by locally constant maps $X \rightarrow A$ and by continuous maps $X \rightarrow K \subseteq A$ for some compact open $K$.

The following notation is used in the formulation of $u_2$. Given two arbitary finite, disjoint sets $\Sigma_1$, $\Sigma_2$ of places of $F$ and a locally profinite group $A$ we put
\[ \mathcal{C}_?(\Sigma_1, A)^{\Sigma_2} = C_?((\mathbb{A}_F^{\Sigma_2})^\ast / U^{\Sigma_1 \cup \Sigma_2}, A). \]
where $? \in \{ \diamond , c, 0 \}$. Here, for a set of places $S$, $U^S$ denotes the subgroup of $\mathbb{A}_F^\ast$ of ideles $(x_v)_v$ with local components $x_v =1$ if $v \in S$, $x_v >0$ if $v \mid \infty$ and $x_v$ is a local unit if $v\nin S \cup R_\infty$.

We also introduce a generalisation of the above notation. For $S_1, S_2$ disjoint sets of places of $F$ let
\[ \mathcal{C}_?(S_1, S_2,A)=C_?( \prod_{\mathfrak{p} \in S_1} F_\mathfrak{p} \times (\mathbb{A}_F^{S_1})^\ast / U^{S_1 \cup S_2}, A). \]
If $S_3$ is an additional disjoint set of places we also define
\[ \mathcal{C}_?(S_1, S_2,A)^{S_3}=C_?( \prod_{\mathfrak{p} \in S_1} F_\mathfrak{p} \times (\mathbb{A}_F^{S_1 \cup S_3})^\ast / U^{S_1 \cup S_2 \cup S_3}, A). \]

\subsection{Measures} \label{s:measure}

We now wish to attach to a homomorphism $\mu : C_c(X, \mathbb{Z}) \rightarrow \mathbb{Z}[G]$ an $A \otimes \mathbb{Z}[G]$-valued measure on $X$ for any abelian group $A$ and finite abelian group $G$. We write the group operation of $A$ multiplicatively. Firstly we note that $\mu$ can be uniquely extended to a homomorphism of $\mathbb{Z}[G]$-modules $\mu : C_c(X, \mathbb{Z}[G]) \cong C_c(X, \mathbb{Z}) \otimes \mathbb{Z}[G] \rightarrow \mathbb{Z}[G] $. By tensoring $\mu$ with the identity map on $A$ we obtain a homomorphism of $\mathbb{Z}[G]$-modules
\begin{equation}\label{e:mua}
    \mu_A :  C_c(X, \mathbb{Z}) \otimes (A \otimes \mathbb{Z}[G] ) \cong C_c^0(X,A \otimes \mathbb{Z}[G]) \rightarrow A\otimes \mathbb{Z}[G].
\end{equation}
To write this map explicitly we first note that the isomorphism in (\ref{e:mua}) is given by 
\[ f \otimes \alpha \mapsto \alpha \cdot f, \ \text{with inverse} \ g \mapsto \sum_{\alpha \in A \otimes \mathbb{Z}[G]} (\alpha \otimes g_\alpha), \]
where $g_\alpha(x)=1$ if $g(x)=\alpha$ and $0$ otherwise. Here we have $f \in  C_c(X, \mathbb{Z})$, $\alpha \in A \otimes \mathbb{Z}[G]$ and $g \in C_c^0(X,A \otimes \mathbb{Z}[G])$. Thus the homomorphism $\mu_A$ is given by
\[ \mu_A(g)=  \sum_{\alpha \in A \otimes \mathbb{Z}[G]} \left( \sum_{\sigma \in G} \sum_{\tau \in G} \alpha_\tau^{\mu_\sigma(g_\alpha)} \otimes \sigma \tau \right) . \]
Where $\alpha = \sum_{\tau \in G} \alpha_\tau \otimes \tau $,  $\mu(g_\alpha)= \sum_{\sigma \in G} \mu_\sigma(g_\alpha) [\sigma] $ and $g_\alpha$ is as defined before. If $A$ is profinite we can consider the homomorphism
\[ \mu_A \coloneqq \varprojlim_K \mu_{A/K} : \varprojlim_K C_c(X, A/K \otimes \mathbb{Z}[G]) \rightarrow \varprojlim_K A/K \otimes \mathbb{Z}[G] = A \otimes \mathbb{Z}[G] \]
where $K$ ranges over the open subgroups of $A$. Since $C_c(X,A \otimes \mathbb{Z}[G]) \subseteq \varprojlim_K C_c(X, A/K \otimes \mathbb{Z}[G])$, we see that $\mu_A$ extends canonically to a homomorphism $C_c(X,A \otimes \mathbb{Z}[G]) \rightarrow A \otimes \mathbb{Z}[G]$ (which we denote by $\mu_A$ as well). For a general $A$ (not necessarily profinite) we have seen that $\mu$ induces a homomorphism $C_c(X,K \otimes \mathbb{Z}[G]) \rightarrow K \otimes \mathbb{Z}[G]$ for every compact open subgroup $K \subset A$. Combining these maps we see that $\mu$ induces a canonical homomorphism $\mu_A : C_c^\diamond (X,A \otimes \mathbb{Z}[G]) \rightarrow A \otimes \mathbb{Z}[G]$. Define the set of $A\otimes \mathbb{Z}[G]$-valued measures on $X$ to be 
\[ \Meas(X,A \otimes \mathbb{Z}[G])= \Hom(C_c^\diamond(X,A\otimes \mathbb{Z}[G]),A\otimes \mathbb{Z}[G]). \]
The map $\mu \mapsto \mu_A$ defines a homomorphism $\Hom(C_c(X, \mathbb{Z}[G]),\mathbb{Z}) \rightarrow \Meas(X,A \otimes \mathbb{Z}[G] )$.

In practice, we apply certain specialisations of the general construction above. In the definition of $u_2$ we construct $\mu \in \Hom(C_c(X, \mathbb{Z}) , \mathbb{Z})$ rather than in $\Hom(C_c(X, \mathbb{Z}) , \mathbb{Z}[G])$. We include $\Hom(C_c(X, \mathbb{Z}) , \mathbb{Z})$ into $\Hom(C_c(X, \mathbb{Z}) , \mathbb{Z}[G])$ by the map
\[ \iota_1 : \Hom(C_c(X, \mathbb{Z}) , \mathbb{Z}) \rightarrow \Hom(C_c(X, \mathbb{Z}) , \mathbb{Z}[G]), \quad  \iota_1(\mu)(f)= \mu(f) [ \id ] ,  \]
for $f \in C_c(X, \mathbb{Z})$.

In the definition of $u_3$ we have a measure on $A$ rather than on $A \otimes \mathbb{Z}[G]$. We include $C_c^\diamond(X,A)$ into $C_c^\diamond(X,A \otimes \mathbb{Z}[G] )$ via the map 
\[ \iota_2 :  C_c^\diamond(X,A) \rightarrow C_c^\diamond(X,A\otimes \mathbb{Z}[G]), \quad  \iota_2(f)(x)= f(x) \otimes \id_G, \]
for $x \in X$.

\subsection{Eisenstein cocycles} \label{s:eisenstein}

We now define the Eisenstein cocycle. The cohomological constructions  $u_2$ and $u_3$ require different variations.

Let $E_+(\mathfrak{f})_\mathfrak{p}$ denote the group of totally positive $\mathfrak{p}$-units of $F$ that are congruent to $1 \pmod{\mathfrak{f}}$. The abelian group $E_+(\mathfrak{f})_\mathfrak{p}$ is free of rank $n$. For $x_1, \dots , x_n \in E_+(\mathfrak{f})_{\mathfrak{p}}$, a fractional ideal $\mathfrak{b}$ coprime to $S$ and $\ell$, and compact open $U \subset F_\mathfrak{p}$, we put
\[ \nu_{\mathfrak{b}, \lambda}^\mathfrak{p}(x_1, \dots , x_n)(U) = \delta(x_1, \dots , x_n) \zeta_{R, \lambda}(\mathfrak{b}, \overline{C}_{e_1}(x_1, \dots, x_n), U,0). \]
Here, the Shintani zeta function is defined in (\ref{shintanizetafn}), $\delta$ is defined in (\ref{epsilonnotation}) and $\overline{C}_{e_1}(x_1, \dots, x_n)$ is defined in (\ref{overlinec}). Then $\nu_{\mathfrak{b}, \lambda}^\mathfrak{p}$ is a homogeneous $(n-1)$-cocycle  on $E_+(\mathfrak{f})_{\mathfrak{p}}$ with values in the space of $\mathbb{Z}$-distributions on $F_\mathfrak{p}$. This follows from \cite[Theorem $2.6$]{MR3351752}. We obtain a class
\[ \omega_{\mathfrak{f}, \mathfrak{b}, \lambda}^\mathfrak{p} \coloneqq [ \nu_{\mathfrak{b}, \lambda}^\mathfrak{p} ] \in H^{n-1}(E_+(\mathfrak{f})_{\mathfrak{p}} , \Hom(C_c(F_\mathfrak{p}, \mathbb{Z}), \mathbb{Z})). \]

\begin{remark} \label{integrationpairing}
The function $\nu_{\mathfrak{b}, \lambda}^\mathfrak{p}(x_1, \dots , x_n)$ is viewed as an element of $\Hom(C_c(F_{S_p}, \mathbb{Z}), \mathbb{Z}[G])$ via the following canonical integration pairing 
\[ (f,\mu) \mapsto \int_{F_R} f(t) d \mu(t) = \lim_{\mid \mid \mathcal{V} \mid \mid \rightarrow 0 } \sum_{V \in \mathcal{V}} f(t_V) \mu(V)  \]
where the limit is over increasingly finer covers $\mathcal{V}$ of the support of $f$ by compact open subgroups $V \subseteq F_{S_p}$ and $t_V \in V$ is any element of $V$.
\end{remark}

We also define
\begin{equation}\label{e: defn of w as a sum}
    \omega_{\mathfrak{f},\lambda}^\mathfrak{p} = \sum_{[\mathfrak{b}] \in G_\mathfrak{f}/ \langle \mathfrak{p} \rangle} \rec_{H/F}(\mathfrak{b}) \omega_{\mathfrak{f}, \mathfrak{b}, \lambda}^\mathfrak{p} \in H^{n-1}(E_+(\mathfrak{f})_{\mathfrak{p}} , \Hom(C_c(F_{\mathfrak{p}}, \mathbb{Z}), \mathbb{Z}[G])),
\end{equation}
where the sum ranges over a system of representatives of $G_\mathfrak{f}/ \langle \mathfrak{p} \rangle$. This construction is adapted from the construction of $\omega_{\mathfrak{f}, \lambda}^\mathfrak{p}$ in \S 3.3 of \cite{MR3968788}.

We write $W$ for $F$ considered as a $\mathbb{Q}$-vector space, and $W_\infty = W \otimes_\mathbb{Q} \mathbb{R}$. As before, let $\lambda$ be a prime of $F$ such that $\N \lambda =\ell$ for a prime number $\ell\in \mathbb{Z}$ and $\ell \geq n+2$. We assume that no primes in $S$ have residue characteristic equal to $\ell$. Let $W_\ell= W \otimes_\mathbb{Q} \mathbb{Q}_\ell$.

Define $\phi_\lambda \in C_c(W_\ell, \mathbb{Z})$ by $\phi_\lambda= \mathbbm{1}_{\mathcal{O}_F \otimes \mathbb{Z}_\ell}- \ell\mathbbm{1}_{\lambda \otimes \mathbb{Z}_\ell}$, i.e.
\begin{equation}
    \phi_\lambda(v)= \begin{cases} 1 & \text{if} \ v \in (\mathcal{O}_F \otimes \mathbb{Z}_\ell)-(\lambda \otimes \mathbb{Z}_\ell), \\
1-\ell &\text{if} \ v \in \lambda \otimes \mathbb{Z}_\ell, \\
0 &\text{if} \ v \in V_\ell -( \mathcal{O}_F \otimes \mathbb{Z}_\ell).
\end{cases}
\label{eqnphi}
\end{equation}

By fixing an ordering of the infinite places $v \in R_\infty$, we fix an identification $W_\infty \cong \mathbb{R}^n$. We define $F^{\ell}_+$ as in \S\ref{s:prelim:notation}. If $D$ is a Shintani set and $\Phi \in C_c(W_{\widehat{\mathbb{Z}}}, \mathbb{Z})$ then, following \cite{MR3861805}, we define the Dirichlet series
\begin{equation}
    L({D}, \Phi;s) = \sum_{v \in W \cap {D}} \Phi(v) N(v)^{-s}.
    \label{Lfnforeis}
\end{equation}
It is known to converge for $\text{Re}(s)>1$ and extend to the whole complex plane except for possibly a simple pole at $s=0$. Moreover, if $D$ and $\Phi$ are as given in the following proposition then $L({D}, \Phi;s)$ is holomorphic. We remark that the set $S$ does not appear in the definition of this Dirichlet series. In the following proposition we will decorate the $L$-function with $\lambda$ since the choice of $\Phi$ incorporates $\lambda$ into it.

\begin{proposition} \label{prop9.4}
Let $\omega_1, \dots , \omega_n \in F^{\ell}_+$. For a map $\phi \in  C_c(W_{\widehat{\mathbb{Z}}^\ell}, \mathbb{Z})$, let
\[\Eis_{F,\lambda}^0(\omega_1, \dots , \omega_n) (\phi)=\delta(\omega_1, \dots , \omega_n ) L_\lambda(\overline{C}_{e_1}(\omega_1, \dots , \omega_n), \Phi; 0), \]
where $\Phi= \phi \otimes  \phi_\lambda$. Then $\Eis_{F,\lambda}^0$ is an $F^{\ell}_+$-homogeneous $(n-1)$-cocycle yielding a class
\[ \Eis_{F,\lambda}^0 \in H^{n-1}(F^{\ell}_+, \Hom(C_c(W_{\widehat{\mathbb{Z}}^\ell},\mathbb{Z}), \mathbb{Z})). \]
\end{proposition}

\begin{proof}
This proposition follows the combination of \cite[Definition 4.5]{MR3861805} and \cite[Lemma 5.1]{MR3861805}.
\end{proof}

We note that in \cite{MR3861805} a more general cocycle $\Eis_{F,\lambda, v}^0$ is constructed. Here $v \in R_\infty$ is a fixed infinite place. We explain the context of this now. For a subgroup $H \subseteq F^{\ell,v}$ and an $H$-module $M$, define $M(\delta)=M \otimes \mathbb{Z}(\delta)$. Thus $M(\delta)$ is the group $M$ with $H$-action given by $x \cdot m = \delta(x)xm$ for $x \in H$ and $m \in M$. Then $\Eis_{F,\lambda, v}^0 \in H^{n-1}(F^{\ell, v}, \Hom(C_c(W_{\widehat{\mathbb{Z}}^\ell},\mathbb{Z}), \mathbb{Z})(\delta))$ and we have the equality
\[ \textrm{res}_{F_+^\ell}^{F^{v,\ell}} ( \Eis_{F,\lambda, v}^0 ) = \Eis_{F,\lambda}^0 . \]

\subsection{Colmez subgroups}\label{s: Colmez subgroups}

In the definitions for the Eisenstein cocycle and its variants the sign map $\delta$ appears. For the explicit calculations we want to perform later it is convenient if we can work with a finite index subgroup $V \subseteq E_+(\mathfrak{f})$ such that $V= \langle g_1 , \dots , g_{n-1} \rangle$ and that we are able to choose $\pi$ such that, after writing $g_n=\pi$,
\begin{itemize}
    \item for $\tau \in S_n$ we have $\delta([g_{\tau(1)} \mid \dots \mid g_{\tau(n-1)} ]) = \sign(\tau)$.
\end{itemize}
We refer to such subgroups as Colmez subgroups. We define
\[ \Log : \mathbb{R}_+^n \rightarrow \mathbb{R}^n, \quad (x_1, \dots , x_n) \mapsto ( \log(x_1), \dots , \log(x_n)). \]
Let $\mathcal{H} \subset \mathbb{R}^n$ be the hyperplane defined by $\text{Tr}(z)=0$. Then, $\Log(E_+(\mathfrak{f}))$ is a lattice in $\mathcal{H}$. If $z = (z_1, \dots , z_n) \in \mathbb{R}_+^n$ and $\Log(z) \in \mathbb{R}^n$ is not an element of $\mathcal{H}$, then we define the projection
\[ z_\mathcal{H} = (z_1 \dots  z_n)^{-\frac{1}{n}} \cdot z . \]
We have $\Log( z_\mathcal{H}) \in \mathcal{H}$. Note that $z$ and $z_\mathcal{H}$ lie on the same ray in $\mathbb{R}_+^n$. For any $M>0$ and $i=0,1, \dots , n-1 $, write $l_i(M)$ for the element of $\mathcal{H}$ which has value $M$ in the $(i+1)$ place and $-M/(n-1)$ in the other places. We endow $\mathbb{R}^n$ with the sup-norm. We denote by $B(x, r)$ the ball centered at $x$ of radius $r$.

The following lemma, which builds on \cite[Lemma 2.1]{MR922806}, allows us to find a collection of possible subsets $V=\langle g_1, \dots , g_{n-1} \rangle$ such that we get a nice sign property that allows us to more easily explicitly calculate the Eisenstein cocycle.

\begin{lemma} \label{lemmatomakembig}
There exists $R_1 >0$ such that for all $R > R_1$, $M >K_1(R)$ (where $K_1(R)$ is some constant we define that depends only on $R$) we have the following: For $i=1, \dots , n-1$ let $g_i \in E_+(\mathfrak{f}) $ and $g_n=g_\pi \in  \pi_\mathcal{H} E_+(\mathfrak{f})$ such that $\Log(g_i) \in B(l_i(M), R) $ and $\Log ( g_\pi) \in  B(l_0(M),R)$. Then
\begin{itemize}
    \item $\langle g_1, \dots , g_{n-1} \rangle \subseteq E_+(\mathfrak{f})$ is a finite index subgroup, and furthermore 
    \item For $\tau \in S_n$ we have $\delta([g_{\tau(1)} \mid \dots \mid g_{\tau(n-1)} ]) = \sign(\tau)$.
\end{itemize}
\end{lemma}

\begin{proof}
This proof largely follows the ideas of Colmez in his proof of \cite[Lemma 2.1]{MR922806}. First, note that both $\Log(E_+(\mathfrak{f})) $ and $ \Log( \pi_\mathcal{H} E_+(\mathfrak{f}))$ are lattices inside $\mathcal{H}$. There exists a constant $R_1 \coloneqq R(E_+(\mathfrak{f}), \pi)$ such that for all $M>0$ and any $r > R(E_+(\mathfrak{f}), \pi)$ there exist $g_1, \dots , g_{n-1} \in  E_+(\mathfrak{f})$ and $g_\pi \in  \pi_\mathcal{H} E_+(\mathfrak{f})$ such that $\Log(g_i) \in B(l_i(M),r)$ for $i=1, \dots , n-1$ and $\Log(g_\pi ) \in B(l_0(M),r)$. The existence of $R_1$ follows from Dirichlet's Unit Theorem and, in particular, the non-vanishing of the regulator of a number field. Since the $l_i(M)$ form a basis of $\mathcal{H}$, the $\Log(g_i)$ form a free family of finite index in $\Log(E_+(\mathfrak{f}))$, if $M$ is large enough relative to $r$, say $M>k(r)$. This gives the first point of the Lemma. It remains to show the second point.

Now take $M$ satisfying:
\begin{enumerate}[i)]
    \item $M \geq 2(n-1)^4 r$,
    \item $M > (n-1)^2 \log (n!)$,
    \item $M >k (r)$.
\end{enumerate}
For simplicity, let $K_1(r)= \max (2(n-1)^4 r, (n-1)^2 \log (n!), k(r))$ so that we only require $M> K_1(r)$.

Write $g_n = g_\pi$ and let $\tau \in S_n$. Denote $\Delta_\tau = \det([g_{\tau(1)} \mid \dots \mid g_{\tau(n-1)}])$. We show $\Delta_\tau >0$. If $\tau$ fixes $n$ then this calculation is covered by the proof of \cite[Lemma 2.1]{MR922806}. 

Suppose $\tau$ is the transposition which swaps $n-1$ and $n$. Then $\Delta_\tau = \det([g_{1} \mid \dots \mid g_{n-2} \mid g_{n} ]) $. Put $E_i = \exp (M( 1-\frac{i-2}{n-1}))$ and $F_i= \exp(-M(\frac{i-1}{n-1}))$. Hence, the matrix given by $[g_{1} \mid \dots \mid g_{n-2} \mid g_{n} ]$ is written
\[ \begin{pmatrix} 1 & \beta_{1,2} F_2 & \beta_{1,3}F_3 & \dots & \beta_{1,n-1}F_{n-1} & \beta_{1,n} E_n \\
1 & \beta_{2,2} E_2 & \beta_{2,3}E_3 & \dots & \beta_{2,n-1}E_{n-1} & \beta_{2,n} E_n \\
1 & \beta_{3,2} F_2 & \beta_{3,3}E_3 & \dots & \beta_{2,n-1}E_{n-1} & \beta_{3,n} E_n \\
1 & \beta_{4,2} F_2 & \beta_{4,3}F_3 & \dots & \beta_{2,n-1}E_{n-1} & \beta_{3,n} E_n \\ \vdots
  &      \vdots           &     \vdots     &  \ddots & \vdots & \vdots \\
1 & \beta_{n-1,2} F_2 & \beta_{n-1,3}F_3 & \dots& \beta_{n-1,n-1} E_{n-1} & \beta_{n-1,n} E_n \\
1 & \beta_{n,2} F_2 & \beta_{n,3}F_3 & \dots & \beta_{n,n-1} F_{n-1} & \beta_{n,n} F_n
\end{pmatrix}, \]
where by i),
\[ e^{\frac{-M}{2(n-1)^3}} < \beta_{i,j} < e^{\frac{M}{2(n-1)^3}}.  \]
Expand $\Delta_\tau $ and isolate the term given by the entries with coordinates $(1,n), (2,2), \dots , (n-1, n-1) , (n,1)$. Using the bounds we defined previously we obtain
\[ \mid \Delta_\tau - e^{\frac{n M}{2}} \beta_{1,n} \prod_{i=2}^{n-1} \beta_{i,i} \mid \leq (n! -1)e^{\frac{M}{2(n-1)^2}}e^{M(\frac{n}{2}-\frac{n}{n-1})}  \]
and so 
\[ \Delta_\tau \geq e^{\frac{nM}{2}}(e^{\frac{-M}{2(n-1)^2}}- (n!-1) e^{(\frac{M}{2(n-1)^2}-\frac{nM}{n-1})}) >0 \]
according to ii). We then show the other required sign properties in the same way.
\end{proof}

It is required in our later calculations to make the following sign calculation.

\begin{lemma}\label{l:basicsign}
For $i=1, \dots, n-1$ let $g_i \in E_+(\mathfrak{f})$ be chosen as in Lemma \ref{lemmatomakembig}. Write $S$ for the $n \times n$ matrix with rows $\Log(g_1), \dots , \Log(g_{n-1}),v_0$ where $v_0 = (1, \dots , 1) \in \mathbb{R}^n$. Then, if $M> 4(n! -1)R$, we have
\[ \sign (\det(S))=(-1)^{n-1}. \]
\end{lemma}

\begin{proof}
We have 
\[ S= \begin{pmatrix} -\frac{M}{n-1}+\beta_{1,1} & M + \beta_{1,2} & -\frac{M}{n-1}+\beta_{1,3} & \dots & -\frac{M}{n-1}+\beta_{1,n} \\
-\frac{M}{n-1}+\beta_{2,1} & -\frac{M}{n-1}+ \beta_{2,2}  & M + \beta_{2,3} & \dots & -\frac{M}{n-1}+ \beta_{2,n}  \\
-\frac{M}{n-1}+ \beta_{3,1} & -\frac{M}{n-1}+ \beta_{3,2}  & -\frac{M}{n-1}+ \beta_{3,3} & \dots & -\frac{M}{n-1}+ \beta_{3,n}  \\ \vdots  &\vdots     &       \vdots    &  \ddots & \vdots \\
-\frac{M}{n-1}+\beta_{n-1,1}   & -\frac{M}{n-1}+ \beta_{n-1,2} & -\frac{M}{n-1}+ \beta_{n-1,3} & \dots & M + \beta_{n-1,n}  \\
1 & 1 & 1 & \dots & 1
\end{pmatrix}, \]
where $-R< \beta_{i,j}  < R$. We now subtract the first column from each of the other columns and expand the determinant along the bottom row. This gives, after letting $B_{i,j}=\beta_{i,j}-\beta_{i,1}$,
\begin{equation}\label{e:matrix2}
    \det S= (-1)^{n-1} \det \begin{pmatrix} \frac{nM}{n-1}+B_{1,2}  &  B_{1,3}  & \dots & B_{1,n}  \\
B_{2,2}  &  \frac{nM}{n-1}+B_{2,3}  & \dots & B_{2,n}   \\ \vdots
  &         \vdots    &     \ddots     & \vdots  \\
B_{n-1,2}    & B_{n-1,2} &  \dots &  \frac{nM}{n-1}+ B_{n-1,n}
\end{pmatrix}.
\end{equation}
Write $S^\prime$ for the matrix in (\ref{e:matrix2}) and note that $-2R < B_{i,j}<2R$, for all $i,j=1, \dots, n-1$. When expanding the determinant of $S^\prime$ and isolating the diagonal terms using the bounds from before, we observe:
\[ \mid \det (S^\prime) - \prod_{i=1}^{n-1} ( \frac{nM}{n-1}+B_{i,i+1} ) \mid \leq ( n! -1) 2R \left( \frac{nM}{n-1} +2R \right)^{n-2}. \]
Thus,
\[ \det S^\prime \geq \left( \frac{nM}{n-1}-2R \right)^{n-1} - ( n! -1) 2R \left( \frac{nM}{n-1} +2R \right)^{n-2}. \]
Since we have assumed $M>4(n!-1)R$ we have
\[ \det S^\prime > \left( \frac{nM}{n-1}-\frac{M}{2(n! -1)} \right)^{n-1} - \frac{M}{2} \left( \frac{nM}{n-1} +\frac{M}{2(n! -1)} \right)^{n-2}. \]
It thus remains to show that for $n \geq 2$ the following holds
\begin{equation}\label{e:rts}
    \left( \frac{n}{n-1}-\frac{1}{2(n! -1)} \right)^{n-1} - \frac{1}{2} \left( \frac{n}{n-1} +\frac{1}{2(n! -1)} \right)^{n-2} >0.
\end{equation}
Firstly, one can see by calculating that the inequality holds for $n=2$. Remarking that as $n$ increases the difference between the two terms in brackets decreases, gives that the value of the left hand side of (\ref{e:rts}) must increase with $n$. Thus (\ref{e:rts}) holds.
\end{proof}

We now let $K_2(R)= \max (K_1(R), 4(n! -1)R)$ so that both Lemma \ref{lemmatomakembig} and Lemma \ref{l:basicsign} hold if $M>K_2(R)$.

\begin{corollary}\label{c:matrixsign}
Let $r>0$ be an integer, $D_+$ an $r \times r$ diagonal matrix with positive entries, $A \in M_{n \times r}(\mathbb{R})$ and $S$ as in Lemma \ref{l:basicsign}. Then the block matrix 
\[ B = \left(
    \begin{array}{c;{2pt/2pt}c}
        A & D_+ \\ \hdashline[2pt/2pt]
        S & 0 
    \end{array}
\right) ,
\]
has determinant of sign $(-1)^{n-1}(-1)^{r(n+r-1)}$.
\end{corollary}

\begin{proof}
Write $d_1, \dots , d_r \in \mathbb{R}_{>0}$ for the diagonal entries of $D_+$. Using cofactor expansion with the last $r$ columns of $B$ one can see that the determinant of $B$ is equal to
\[ \det (S) \prod_{i=1}^{r}d_i (-1)^{(n+r-i)+(i-1))} = \det(S) (-1)^{r(n+r-1)}\prod_{i=1}^{r}d_i .  \]
Using Lemma \ref{l:basicsign} and the fact that the entries of $D_+$ are positive, the result follows. 
\end{proof}

We recall the definition of $k(r)$ from the proof of Lemma \ref{lemmatomakembig} and note the following lemma.

\begin{lemma}\label{l:scalek}
We can choose $k(r) = K r$ where $K$ is some constant that does not depend on $r$. I.e., suppose $r>R_1$ and $M>Kr$, if for $i=1, \dots , n-1$, we have $g_i \in E_+(\mathfrak{f})$ with $\Log (g_i ) \in B( l_i(M),r)$ then the $\Log (g_i)$ form a free family of finite index in $\Log (E_+(\mathfrak{f}))$.
\end{lemma}

\begin{proof}
We begin by noting that the result is trivial if $n=2$. Suppose that $n>2$. We claim that it is enough to take $K=2(n-1)$. For each $i=1, \dots , n-1$, let $\Log(g_i) \in B( l_i(M),r)$. We then write
\[ \Log(g_i) = (\alpha_i(1), \dots , \alpha_i(n)) \in \mathcal{H} . \]
It is enough to show that the $\Log(g_i)$ are linearly independent under the projection
\begin{align*}
    \varphi : \mathcal{H} & \rightarrow \mathbb{R}^{n-1} \\
    (\alpha_1, \dots , \alpha_n) & \mapsto (\alpha_1, \dots , \alpha_{n-1}).
\end{align*}
By the definition of $l_i(M)$ and our choice of $r$ it is clear that 
\[ \alpha_i(j) > 0 \text{ if } j=i+1 \quad \text{and} \quad \alpha_i(j) < 0 \text{ otherwise. }  \]
We note that $\alpha_{n-1}(j)<0$ for all $j$. It follows immediately that the vectors
\[ \varphi(\Log(g_1)), \dots , \varphi(\Log(g_{n-1})) \]
are linearly independent. Thus the $\Log(g_i)$ for a free family of finite index in $\Log(E_+(\mathfrak{f}))$.

\end{proof}

It follows from the above lemma that if $M>K_2(R)$ then for any $\lambda>1$ we have that $\lambda M > K_2(\lambda R)$.

\begin{lemma}\label{l:coprimechoices}
Let $R_1>0$ be as is shown to exist in Lemma \ref{lemmatomakembig}. There exists
\begin{enumerate}
    \item $R_f, R_g > R_1$,
    \item $M_f > K_2(R_f)$ and
    \item $M_g > K_2(R_g)$,
\end{enumerate}
such that we have the following. Firstly, for $i=1, \dots , n-1$ we can choose $f_i, g_i \in E_+(\mathfrak{f})$ such that $\Log( f_i) \in  B(l_i(M_f), R) $ and $ \Log ( g_i ) \in B(l_i(M_g), R) $. Furthermore, after writing
\[ V_f = \langle f_1, \dots , f_{n-1} \rangle \quad \text{and} \quad V_g = \langle g_1, \dots , g_{n-1} \rangle \]
we have that $[E_+(\mathfrak{f}) : V_f]$ is coprime to $[E_+(\mathfrak{f}) : V_g]$.
\end{lemma}

\begin{proof}
We firstly choose the $f_i \in E_+(\mathfrak{f})$ via Lemma \ref{lemmatomakembig} and Lemma \ref{l:basicsign}, and let $V_f = \langle f_1, \dots , f_{n-1} \rangle$. I.e., we have $\Log(f_i) \in  B(l_i(M_f), R_f) $ for some $R_f>R_1$ and $M_f > K_2(R_f)$. 

By writing the matrix representing the generators we have chosen for $V_f$ in an upper triangular form, we can make the following choice of generators of $E_+(\mathfrak{f})$. Let $ \langle \delta_1, \dots , \delta_{n-1} \rangle =E_+(\mathfrak{f}) $ such that for some $\tau \in S_{n-1}$ we have, for $i=1, \dots , n-1$,
\[ f_{\tau(i)} = \delta_i^{a_i} \prod_{j=1}^{i-1}\delta_j^{b_{i,j}} , \]
and $[E_+(\mathfrak{f}):V_f]=  \prod_{i=1}^{n-1} \mid a_i \mid $. By changing the sign if necessary we choose $a_1>0$. Furthermore, we note that changing the values of the $b_{i,j}$ in the choice of $V_f$ does not change the index of the subgroup.

For ease of notation, let $a= \prod_{i=1}^{n-1} \mid a_i \mid$. For $i=2, \dots , n-1$ there exists $R_{g,i}>0$ and $M_{g,i}>0$ such that for all $M>M_{g,i}$, there exists $\alpha \in E_+(\mathfrak{f})$ with $ \Log(\alpha) \in B(l_{\tau(i)}(M), R_{g,i})$ and
\[ \alpha = \delta_i^{q_i} \prod_{j=1}^{i-1}\delta_j^{k_j}, \]
with $q_i$ a nonzero integer with absolute value coprime to $a$. We note that this is only possible for $i>2$ since we require the freedom of having at least one additional component we can vary. 

We now consider $i=1$. We have $\Log ( f_{\tau(1)})= \Log( \delta_1^{a_1}) \in B(l_1(M_f), R_f) $. Therefore any $q_1>a_1$ we have $\Log( \delta_1^{q_1}) \in B( \frac{q_1}{a_1}l_1(M_f),\frac{q_1}{a_1} R_f) $.

Now let $R_g^\prime = \max (R_1, R_{g,2} , \dots, R_{g,n-1} )$ and $M_g^\prime = \max (  M_{g,2}, \dots , M_{g, n-1})$. We now find $q_1>a_1$ which is coprime to $a$ and such that $\frac{q_1}{a_1}M_f>M_g^\prime$ and $\frac{q_1}{a_1}R_f>R_g^\prime$. 

We now fix $R_g = \frac{q_1}{a_1}R_f$ and $M_g= \frac{q_1}{a_1}M_f$. Clearly $R_g> R_1$ and it follows from Lemma \ref{l:scalek} that $M_g>K_2(R_g)$. We then choose $g_{\tau(1)}= \delta_1^{q_1}$, it is immediate that $\Log(g_{\tau(1)}) \in B(l_1(M_g), R_g)$. For $i=2, \dots , n-1$, we have shown that there exist $g_{\tau(i)} \in E_+(\mathfrak{f})$ with $ \Log(g_{\tau(i)}) \in B(l_{\tau(i)}(M_g), R_{g})$ and
\[ g_{\tau(i)} = \delta_i^{q_i} \prod_{j=1}^{i-1}\delta_j^{k_j}, \]
with $q_i$ a nonzero integer with absolute value coprime to $a$. Let $V_g = \langle g_1 , \dots , g_{n-1} \rangle $, the result follows.
\end{proof}

\subsection{1-cocycles attached to homomorphisms} \label{s:1cocycle}

Let $g\colon F_\mathfrak{p}^\ast \rightarrow A$ be a continuous homomorphism, where $A$ is a locally profinite group. We now define a cohomology class $c_g \in H^1(F_\mathfrak{p}^\ast , C_c (F_\mathfrak{p}, A))$ attached to $g$. The $F_\mathfrak{p}^\ast$-action on $C_c(F_\mathfrak{p}^\ast, \mathbb{Z})$ is defined by $(xf)(y)=f(x^{-1}y)$. The following definition is due to the third author and first appears in \cite[Lemma $2.11$]{MR3179573}. This definition is crucial in making the constructions of the first and third authors cohomological formulas work. We also remark that the definition is unusual in that it appears as though the cocycle $z_g$ should be a coboundary. However, it may not be a coboundary since $g$ does not necessarily extend to a continuous function on $F_\mathfrak{p}$.
 
\begin{definition} \label{1cocycle}
Let $g:F_\mathfrak{p}^\ast \rightarrow A$ be a continuous homomorphism, where $A$ is a locally profinite group. Let $f \in C_c(F_\mathfrak{p},\mathbb{Z})$ such that $f(0)=1$. We define $c_g$ to be the class of the cocycle
\[ z_{f,g}:F_\mathfrak{p}^\ast \rightarrow C_c (F_\mathfrak{p},A) \]
defined by $z_{f,g} (x)=``(1-x)(g \cdot f)"$, or more precisely
\begin{equation}
    z_{f,g} (x)(y)=(xf)(y) \cdot g(x) + ((f-xf)\cdot g)(y)
    \label{zeqn}
\end{equation}
for $x \in F_\mathfrak{p}^\ast$ and $y \in F_\mathfrak{p}$.
\end{definition}

The second term in (\ref{zeqn}) is allowed to be evaluated at $0 \in F_\mathfrak{p}$ since we can extend continuously the function from $F_\mathfrak{p}^\ast$ to ${F}_\mathfrak{p}$ as 
\[(f-xf)(0)=0.\]
The class $c_g = [ z_{f,g} ] \in H^1(F_\mathfrak{p}^\ast , C_c(F_\mathfrak{p}, A))$ is independent of the choice of $f \in C_c( F_\mathfrak{p}, \mathbb{Z})$ with $f(0)=1$. In particular, we can consider the class $c_\id \in H^1(F_\mathfrak{p}^\ast, C_c(F_\mathfrak{p},F_\mathfrak{p}^\ast))$. For more details on this construction, see \cite[\S 3.2]{MR3861805} and \cite[\S 3.1]{MR3968788}.

\subsection{Homology of a group of units}\label{s:homologyunits}

Let $V \subseteq E_+$ be a finite index subgroup. Recall we have written $G_\mathfrak{f}$ for the narrow ray class group of conductor $\mathfrak{f}$. Let $e$ be the order of $\mathfrak{p}$ in $G_\mathfrak{f}$, and write $\mathfrak{p}^{e} = (\pi)$ with $\pi \equiv 1 \pmod{\mathfrak{f}}$ and $\pi$ totally positive. Write $V_\mathfrak{p} = V \oplus \langle \pi \rangle$. 

By Dirichlet's unit theorem, the group $V_\mathfrak{p}$ is free abelian of rank $n$. Thus the homology groups $H_{n}(V_{\mathfrak{p}}, \mathbb{Z})$ is free abelian of rank $1$. In the comological formulas $u_2$ and $u_3$, we are required to choose a generator of this homology group.  For these two invariants, we will be working in the cases $V=E_+$ and $V=E_+(\mathfrak{f})$, respectively. 

Write $V = \langle \varepsilon_1, \dots , \varepsilon_{n-1} \rangle$. For ease of notation, write $\pi = \varepsilon_{n}$. We then choose the following generator for the group $H_{n}(V_{\mathfrak{p}},  \mathbb{Z})$,
\begin{equation}\label{e:unitgen}
    \eta_{\mathfrak{p}} = \mu \sum_{\tau \in S_{n}} \sign(\tau)[\varepsilon_{\tau(1)} \mid \dots \mid \varepsilon_{\tau(n)}]\otimes 1.
\end{equation}
Here $\mu \in \{ 1, -1 \}$ and is equal to the sign of the determinant of a specific matrix. For $x \in E_{\mathfrak{p},+}$, let 
\[ L(x) =  
(\log (\sigma_1(x)), \dots , \log (\sigma_{n}(x)) , \ord_{\mathfrak{p}}(x) ) .
  \]
Define $L_1 \in \mathbb{R}^{n}$ to be the vector with $1$ in the first $n$ components and $0$ in the last component. 
Then $\mu$ is the sign of the determinant of the matrix with rows 
\[ L(\pi) , L(\varepsilon_1), \dots , L(\varepsilon_{n-1}) , L_1. \]
This choice generalises that given in \cite[Remark 2.1]{MR3201900}.

\section{Cohomological formula I ($u_2$)} \label{s:u2}

This section follows the construction given in \cite[\S 3.1]{MR3861805}. For ease of notation and to reduce the exposition of this section we give a simpler definition for $u_2$ than appears in \cite[\S 3.1]{MR3861805}. In particular, we do not involve the infinite places in the construction we give. For our purposes this definition is enough and it simplifies the arguments in \S \ref{s:u23}. Throughout this section we use the notation established in \S\ref{s:cont}. 

Let $\eta_\mathfrak{p}$ be the generator of $H_n(E_{\mathfrak{p}, +}, \mathbb{Z})$ defined in (\ref{e:unitgen}) and let $\mathcal{F}$ be a fundamental domain for the action of $F^{\ell}_+ / E_{\mathfrak{p},+}$ on $(\mathbb{A}_F^{\mathfrak{p}, \ell, \infty})^\ast / U^{\mathfrak{p}, \ell, \infty}$. Then $\mathbbm{1}_\mathcal{F} $ is an element of $ H^0(E_{\mathfrak{p},+} , C(\mathcal{F}, \mathbb{Z})) = C(\mathcal{F}, \mathbb{Z})^{E_{\mathfrak{p},+}}$. Taking the cap product gives $\mathbbm{1}_\mathcal{F} \cap \eta_\mathfrak{p} \in H_{n}(E_{\mathfrak{p},+}, C(\mathcal{F}, \mathbb{Z}))$. We now define $\vartheta^\mathfrak{p} \in H_n(F^\ast , \mathcal{C}_c(\emptyset, \mathbb{Z})^{\mathfrak{p}, \ell, \infty})$ as the homology class corresponding to $\mathbbm{1}_\mathcal{F} \cap \eta_\mathfrak{p}$ under the isomorphism 
\begin{equation}
    H_n(E_{\mathfrak{p},+}, C(\mathcal{F}, \mathbb{Z})) \cong H_n(F^\ell_+ , C_c((\mathbb{A}_F^{\mathfrak{p}, \ell , \infty})^\ast / U^{\mathfrak{p}, \ell, \infty}, \mathbb{Z})) 
    \label{isom1}
\end{equation}
that is induced by $C_c((\mathbb{A}_F^{\mathfrak{p}, \ell , \infty})^\ast / U^{\mathfrak{p}, \ell, \infty}, \mathbb{Z}) \cong \text{Ind}_{E_{\mathfrak{p}, +}}^{F^\ell_+} C(\mathcal{F}, \mathbb{Z})$. 

We now follow the construction of \cite[\S 6]{MR3861805}. Since the local norm residue symbol for $H/F$ at $\mathfrak{p}$ is trivial we omit it from the reciprocity map, i.e. we consider the homomorphism
\[ \rec_{H/F}^{\mathfrak{p}, \ell, \infty} : (\mathbb{A}_F^\mathfrak{p})^\ast / U^{R, \mathfrak{p}, \ell} \rightarrow G \hookrightarrow \mathbb{Z}[G]^\ast, \quad x=(x_v)_{v \neq \mathfrak{p}} \mapsto \prod_{v \not \in \{ \mathfrak{p} , \ell \} \cup R } (x, H/F)_v . \]
Let $R^\prime = R-R_\infty$. We can view $\rec_{H/F}^{\mathfrak{p}, \ell, \infty}$ as an element of $H^0(F^\ell_+ , \mathcal{C}_c(R^\prime, \mathbb{Z}[G])^{\mathfrak{p} , \ell , \infty})$ and denote by
\[ \rho_{H/F} \in H_n(F^\ell_+, \mathcal{C}_c(R^\prime, \mathbb{Z}[G])^{\mathfrak{p} , \ell , \infty} ) \]
its image under the map
\[ H^0(F^\ell_+ , \mathcal{C}_c(R^\prime, \mathbb{Z}[G])^{\mathfrak{p} , \ell , \infty}) \rightarrow H_n(F^\ell_+, \mathcal{C}_c(R^\prime, \mathbb{Z}[G])^{\mathfrak{p} , \ell , \infty}), \quad \psi \mapsto \psi \cap \vartheta^\mathfrak{p}. \]
Here the cap product is induced by the map
\begin{equation}
    \mathcal{C}^\diamond (R^\prime, \mathbb{Z}[G])^{\mathfrak{p} , \ell , \infty} \times \mathcal{C}_c(\emptyset, \mathbb{Z})^{\mathfrak{p} } \rightarrow \mathcal{C}_c^\diamond (R^\prime, \mathbb{Z}[G])^{\mathfrak{p} , \ell , \infty}  , \quad (\psi, \phi) \mapsto \psi \cdot \phi,
    \label{map1}
\end{equation}
here $\psi \cdot \phi$ denotes the function $xU^{R^\prime \cup \{ \mathfrak{p}, \ell , \infty \}} \mapsto \psi(xU^{R^\prime \cup \{ \mathfrak{p}, \ell , \infty \}}) \phi(xU^{\mathfrak{p}})$. 

For a locally profinite abelian group $A$ we have a canonical map
\[ C_c^\diamond(F_\mathfrak{p},A) \otimes \mathcal{C}_c(R^\prime, \mathbb{Z}[G])^{\mathfrak{p} , \ell , \infty} \rightarrow \mathcal{C}_c^\diamond(\mathfrak{p}, R^\prime, A \otimes \mathbb{Z}[G])^{\ell, \infty}, \quad (f,g) \mapsto f \otimes g, \]
which induces a cap-product pairing
\[ H^1(F^\ast , C_c^\diamond(F_\mathfrak{p},A)) \times H_n(F^\ast, \mathcal{C}_c(R^\prime, \mathbb{Z}[G])^{\mathfrak{p} , \ell , \infty}) \rightarrow H_{n-1}(F^\ast , \mathcal{C}_c^\diamond(\mathfrak{p}, R^\prime, A \otimes \mathbb{Z}[G])^{\ell, \infty}). \]
In particular we can consider
\[ c_\id \cap \rho_{H/F} \in H_{n-1}(F^\ast , \mathcal{C}_c^\diamond(\mathfrak{p}, R^\prime, F_\mathfrak{p}^\ast \otimes \mathbb{Z}[G])^{\ell, \infty}). \]
Here $c_\id$ is as defined in Definition \ref{1cocycle}. Recall that we write $W$ for $F$ considered as a $\mathbb{Q}$-vector space. In \cite[\S 5.3]{MR3861805}, the following map is defined. 
\[ \Delta_\ast : H_{n-1} (F^\ell_+ , \mathcal{C}^\diamond_c (\mathfrak{p}  , R^\prime, F_\mathfrak{p}^\ast \otimes \mathbb{Z}[G] )^{\ell, \infty}   )  \rightarrow H_{n-1} (F^{\ell}_+, C^\diamond_c (W_{\widehat{\mathbb{Z}}^\ell} , F_\mathfrak{p}^\ast \otimes \mathbb{Z}[G] )). \]
We postpone giving the definition of $\Delta_\ast$ until the next section.

Now consider the canonical pairing, where we recall the definition of $\mu_{F_\mathfrak{p}^\ast}$ from \S\ref{s:measure},
\begin{equation}\label{e:measurecap}
    \Hom(C_c(W_{\widehat{\mathbb{Z}}^\ell}, \mathbb{Z}), \mathbb{Z}) \times C_c^\diamond (W_{\widehat{\mathbb{Z}}^\ell}, F_\mathfrak{p}^\ast \otimes \mathbb{Z}[G]) \rightarrow F_\mathfrak{p}^\ast \otimes \mathbb{Z}[G], \quad (\mu, f) \mapsto \mu_{F_\mathfrak{p}^\ast}(f).
\end{equation}
Noting that $F^{\ell }_+$ is acting trivially on $ F_\mathfrak{p}^\ast \otimes \mathbb{Z}[G] $ we see that (\ref{e:measurecap}) induces, via cap-product, a pairing
\begin{equation}
    \cap : H^{n-1}(F^{\ell}_+, \Hom(C_c (W_{\widehat{\mathbb{Z}}^\ell}, \mathbb{Z}), \mathbb{Z})) \times H_{n-1}(F^{\ell}_+, C_c^\diamond (W_{\widehat{\mathbb{Z}}^\ell}, F_\mathfrak{p}^\ast \otimes \mathbb{Z}[G])) \rightarrow F_\mathfrak{p}^\ast \otimes \mathbb{Z}[G] .
    \label{finalcapprod}
\end{equation}
Recall the Eisenstein cocycle, $\Eis_{F,\lambda}^0$, from Proposition \ref{prop9.4}. Applying (\ref{finalcapprod}) with the Eisenstein cocycle $ \Eis_{F,\lambda}^0$ and $\Delta_\ast(c_\id \cap \rho_{H/F})$ we obtain the element of $ F^\ast_\mathfrak{p} \otimes \mathbb{Z}[G]$, defined in \cite[\S 3.1]{MR3861805}. Therefore,
\begin{equation}
   u_2 = u({S},\lambda)= \sum_{\sigma \in G} u_2(\sigma) \otimes [\sigma^{-1}]=\Eis_F^0 \cap \Delta_\ast(c_\id \cap \rho_{H/F}).
   \label{defneqnfforu2}
\end{equation}
The first and third authors then conjecture that the element $u_2(\sigma)$ is equal to the image of the Brumer--Stark unit in $F_\mathfrak{p}^\ast$ under $\sigma$. We end this section by stating some known properties of this construction.

\begin{remark}
    As noted at the start of this section, the definition of $u_2$ given above is equivalent to that given in \cite[\S 3.1]{MR3861805}. This follows from standard properties of the cap-product.
\end{remark}

\begin{proposition}[Proposition 6.3, \cite{MR3861805}] \label{prop6.3}
\begin{enumerate}[a)]
    \item For $\sigma \in G$ we have $\ord_\mathfrak{p}(u_2(\sigma)) = \zeta_{R,T}(\sigma, 0)$.
    
    \item Let $L/F$ be an abelian extension with $L \supseteq H$ and put $\mathfrak{g}= \Gal(L/F)$. Assume that $L/F$ is unramified outside $S$ and that $\mathfrak{p}$ splits completely in $L$. Then we have
    \[ u_2(\sigma) = \prod_{\tau \in \mathfrak{g}, \tau \mid_H = \sigma} u_2(L/F , \tau). \]
    
    \item Let $\mathfrak{r}$ be a nonarchimedean place of $F$ with $\mathfrak{r} \nin S \cup \overline{T}$ where $\overline{T}$ is as defined in (\ref{e:sbar}). Then we have 
    \[ u_2(S \cup \{ \mathfrak{r} \} , \sigma) = u_2(S ,  \sigma) u_2(S , \sigma_\mathfrak{r}^{-1} \sigma)^{-1}. \]
    
    \item Assume that $H$ has a real archimedean place. Then $u_2( \sigma) =1$ for all $\sigma \in G$.
    
    \item Let $L/F$ be a finite abelian extension of $F$ containing $H$ and unramified outside $S $. Then we have 
\[ \rec_\mathfrak{p}(u_2( \sigma))= \prod_{\substack{\tau \in \Gal(L/F) \\ \tau \mid_H = \sigma^{-1}  }} \tau^{\zeta_{S,T}(L/F, \tau^{-1},0)}. \]
\end{enumerate}
\end{proposition}

\begin{remark}
In the proposition above we correct a small typo in \cite[Proposition 6.3, c)]{MR3861805} by replacing $\sigma_\mathfrak{r}$ with $\sigma_\mathfrak{r}^{-1}$.
\end{remark}

\subsection{The map $\Delta_\ast$} \label{s:delta}

We now define the map $\Delta_\ast$. For more information and the more general construction we refer to \cite[\S 5.3]{MR3861805}.
Throughout this section we let $A=F_\mathfrak{p}^\ast \otimes \mathbb{Z}[G]$ to ease notation. For sets $X_1, X_2$ and a map $\psi\colon X_1 \times X_2 \rightarrow A$, we write
\[ \Supp(X_1, X_2, \psi) \coloneqq \{ x_1 \in X_1 \mid \exists \ x_2 \in X_2 \ \text{with} \ (x_1, x_2) \in \text{supp}(\psi) \}. \]
Where $\text{supp}(\psi)$ is the support of $\psi$.
\begin{proposition} \label{propctsmaps}
Let $X_1, X_2$ be totally disconnected topological Hausdorff spaces, with $X_1$ discrete. Let $A$ be a locally profinite group. The map
\begin{equation}\label{e:cisom} C_c(X_1, \mathbb{Z}) \otimes_\mathbb{Z} C_c^\diamond (X_2, A) \rightarrow C_c^\diamond (X_1 \times X_2, A), 
\end{equation}
\[ f \otimes g \mapsto ((x_1, x_2) \mapsto f(x_1) \cdot g(x_2)) \]
is an isomorphism.
\end{proposition}

\begin{proof}
We calculate the inverse map as follows. For $\psi \in C_c^\diamond (X_1 \times X_2, A)$ we write $Y_1(\psi)= \Supp(X_1, X_2, \psi) \subseteq X_1$. Note that $Y_1(\psi)$ is finite since $\psi$ has compact support. Then
\[ \psi \mapsto \sum_{y \in Y_1(\psi)} \mathbbm{1}_y \otimes_\mathbb{Z} \psi(y,\cdot ) \in C_c(X_1, \mathbb{Z}) \otimes_\mathbb{Z} C_c^\diamond (X_2, A) \]
provides an inverse to (\ref{e:cisom}).
\end{proof}

We now construct the $F^\ell_+$-equivariant map 
\begin{equation}\label{e: delta map}
    \Delta : \mathcal{C}_c^\diamond (\{ \mathfrak{p} \},R^\prime,A)^{\ell, \infty} \rightarrow C_c^\diamond(\mathbb{A}_F^{\ell, \infty}, A) \cong C_c^\diamond (W_{\widehat{\mathbb{Z}}^\ell},A).
\end{equation}
Recall that we have written $S^\prime = R^\prime \cup \{ \mathfrak{p} \}$ and $\mathbb{A}_F^{\ell, \infty} \cong W_{\widehat{\mathbb{Z}}^\ell}$. There exist canonical homomorphisms 
\begin{equation}
    C_c^\diamond (  F_\mathfrak{p} \times \prod_{\mathfrak{q} \in R^\prime} F_\mathfrak{q}^\ast , A) \otimes \mathcal{C}_c(\emptyset, \mathbb{Z})^{S^\prime \cup \ell, \infty} \rightarrow \mathcal{C}_c^\diamond (\{ \mathfrak{p} \},R^\prime,A)^{\ell, \infty} , \label{isomfordelta}
\end{equation}
\begin{equation}
    C_c^\diamond ( \prod_{\mathfrak{q} \in S^\prime} F_\mathfrak{q}  , A) \otimes C_c(\mathbb{A}^{S^\prime \cup \ell, \infty}_F, \mathbb{Z}) \rightarrow C_c^\diamond (\mathbb{A}^{\ell, \infty}_F,A). \label{isomfordelta2}
\end{equation}
It follows from Proposition \ref{propctsmaps} that the map (\ref{isomfordelta}) is an isomorphism. Let $\mathcal{I}^{S^\prime \cup \ell}$ denote the group of fractional ideals of $F$ that are coprime to $S^\prime \cup \ell$. Since $(\mathbb{A}_F^{S^\prime \cup \ell, \infty})^\ast / U^{S^\prime \cup \ell, \infty}$ is isomorphic to $\mathcal{I}^{S^\prime \cup \ell}$, the ring $\mathcal{C}_c^0(\emptyset, \mathbb{Z})^{S^\prime \cup \ell, \infty}$ can be identified with the group ring $\mathbb{Z}[\mathcal{I}^{S^\prime \cup \ell}]$. We define (\ref{e: delta map}) as the tensor product $\Delta = i \otimes I^{S \cup \ell}$ where $i : C_c^\diamond (  F_\mathfrak{p} \times \prod_{\mathfrak{q} \in R^\prime} F_\mathfrak{q}^\ast , A) \rightarrow  C_c^\diamond ( \prod_{\mathfrak{q} \in S^\prime} F_\mathfrak{q}  , A)$ is the inclusion map induced by extension by $0$ and $I^{S^\prime \cup \ell} : \mathbb{Z}[\mathcal{I}^{S^\prime \cup \ell}] \rightarrow C_c(\mathbb{A}_F^{S^\prime \cup \ell, \infty}, \mathbb{Z})$ maps a fractional ideal $\mathfrak{a} \in \mathcal{I}^{S^\prime \cup \ell}$ to the characteristic function of $\widehat{\mathfrak{a}}^{S^\prime \cup \ell} = \mathfrak{a}(\prod_{\mathfrak{p} \nin S^\prime \cup \ell} \mathcal{O}_\mathfrak{p})$. Considering the map in (\ref{isomfordelta2}) completes our construction of $\Delta$.

\section{Cohomological formula II ($u_3$)} \label{s:u3}

In \cite{MR3968788} the first and third authors give two equivalent constructions for their formula. In this work, we only require the construction given in \cite[\S 3.3]{MR3968788}, which we denote by $u_3$. We refer readers to \cite[\S 3]{MR3968788} for the other formula.

Recall that in \S\ref{s:1cocycle} and \S\ref{s:eisenstein} we have defined the following objects:
\[ c_{\id} \in H^1(F_\mathfrak{p}^\ast, C_c(F_\mathfrak{p},F_\mathfrak{p}^\ast)) \quad \text{and} \quad \omega_{\mathfrak{f}, \lambda}^\mathfrak{p}\in H^{n-1}(E_+(\mathfrak{f})_{\mathfrak{p}} , \Hom(C_c(F_\mathfrak{p}, \mathbb{Z}), \mathbb{Z}[G])).  \]

\begin{definition} \label{u3seconddefn}
Let $ \eta_{\mathfrak{p},E_+(\mathfrak{f}) }   \in H_{n}(E_+(\mathfrak{f})_{\mathfrak{p}}, \mathbb{Z})$ be the generator defined in (\ref{e:unitgen}). Then, we define
\begin{equation}
    u_3 \coloneqq (-1)^{n+1} ( c_{\id} \cap (\omega_{\mathfrak{f},\lambda}^\mathfrak{p} \cap  \eta_{\mathfrak{p},E_+(\mathfrak{f}) }) ) \in F_\fp^\ast \otimes \Z[G] .
\end{equation}
\end{definition}

As noted in the introduction we have modified the definition from \cite{MR3968788} by multiplying by $(-1)^{n+1}$, namely, if we let $u_3^\prime$ be the element defined in \cite{MR3968788} then $u_3=(-1)^{n+1} u_3^\prime$. Adapted from \cite[Conjecture 3.1]{MR3968788} we have the following conjecture.

\begin{conjecture} \label{conju3}
We have $u_3= u_{\mathfrak{p}}$.
\end{conjecture}

\subsection{Transferring to a subgroup} \label{s:transfer} 

Let $V$ be a finite index subgroup of $E_+(\mathfrak{f})$ and $\mathfrak{b}$ a fractional ideal coprime to $S$ and $\ell$. Let $\eta_{\mathfrak{p}, V} \in H_n(V \oplus \langle \pi \rangle, \mathbb{Z})$ be the generator defined in (\ref{e:unitgen}). For $x_1, \dots , x_n \in V \oplus \langle \pi \rangle $ and compact open $U \subset F_\mathfrak{p}$ we put 
\[ \nu_{\mathfrak{b}, \lambda,V}^\mathfrak{p}(x_1, \dots , x_n)(U) \coloneqq \delta(x_1, \dots , x_n) \zeta_{R, \lambda}(\mathfrak{b},\overline{C}_{e_1}(x_1, \dots , x_n),U,0). \]
As before, it follows from \cite[Theorem 2.6]{MR3351752} that $\nu_{\mathfrak{b}, \lambda,V}^\mathfrak{p}$ is a homogeneous $(n-1)$-cocycle on $V \oplus \langle \pi \rangle$ with values in the space of $\mathbb{Z}$-distribution on $F_\mathfrak{p}$. Hence we obtain a class
\[ \omega_{\mathfrak{f}, \mathfrak{b}, \lambda,V}^\mathfrak{p} \coloneqq [ \nu_{\mathfrak{b}, \lambda,V}^\mathfrak{p} ] \in H^{n-1}(V \oplus \langle \pi \rangle ,  \Hom(C_c(F_\mathfrak{p}, \mathbb{Z}), \mathbb{Z})).  \]
We then define
\[ u_3(V,\sigma_\mathfrak{b}) = (-1)^{n+1} ( c_{\id} \cap (\omega_{\mathfrak{f}, \mathfrak{b}, \lambda,V}^\mathfrak{p} \cap \eta_{\mathfrak{p}, V} ) ) , \]
and write 
\[ u_3(V) = \sum_{\sigma \in G} u_3(V, \sigma) \otimes \sigma \in F_\fp^\ast \otimes \Z[G] . \]
The next proposition shows the relation between $u_3$ and $u_3(V)$. Here we adopt the convention that for an element $x=\Sigma_{\sigma \in G} x_\sigma \otimes \sigma \in F_\fp^\ast \otimes \Z[G]$ and $k \in \Z$ we write
\[ x^k = \sum_{\sigma \in G} x_\sigma^k \otimes \sigma \in F_\fp^\ast \otimes \Z[G] . \]

\begin{proposition}[Proposition 6.12, \cite{comparingformulas}] \label{changetoVu3}
Let $V$ be a finite index subgroup of $E_+(\mathfrak{f})$. Then we have
\begin{equation}
    u_3(V) =  u_3^{[E_+(\mathfrak{f}) : V]}.
\end{equation}
\end{proposition}

\subsection{Explicit expression for $u_3$} \label{s:eeu3p}

For later calculations we require an explicit expression for $u_3(V)$ for an appropriate choice of $V \subset E_+(\mathfrak{f})$. Let $V$ be a finite index subgroup of $E_+(\mathfrak{f})$ such that $V= \langle \varepsilon_1 , \dots , \varepsilon_{n-1} \rangle$, where $\varepsilon_1 , \dots , \varepsilon_{n-1}$ and $\pi = \varepsilon_n$ are chosen to satisfy Lemma \ref{lemmatomakembig} and Lemma \ref{l:basicsign}. For $i=1, \dots ,n$ write 
\[ \mathcal{B}_i \coloneqq  \bigcup_{\substack{\tau \in S_{n} \\ \tau(n)=i}} \overline{C}_{e_1}([\varepsilon_{\tau (1)} \mid \dots \mid  \varepsilon_{\tau(n-1)} ]). \]
Let $\mathcal{B}=\mathcal{B}_n$.

\begin{lemma}
    Let $V$ be a finite index subgroup of $E_+(\mathfrak{f})$ such that $V= \langle \varepsilon_1 , \dots , \varepsilon_{n-1} \rangle$, where $\varepsilon_1 , \dots , \varepsilon_{n-1}$ and $\pi = \varepsilon_n$ are chosen to satisfy Lemma \ref{lemmatomakembig} and Lemma \ref{l:basicsign}. Then, for $\sigma \in G$, we have
    \begin{equation}\label{e:u3 explicit}
        u_3(V, \sigma) =   \prod_{i=1}^{n-1} \varepsilon_i^{\zeta_{R,\lambda}(\mathfrak{b}, \mathcal{B}_i,\pi \mathcal{O}_\mathfrak{p},0)} \pi^{\zeta_{R,\lambda}(\mathfrak{b}, \mathcal{B},\mathcal{O}_\mathfrak{p},0) }   \multint_{\mathbb{O}} x \ d( \zeta_{R,\lambda}(\mathfrak{b}, \mathcal{B},x,0) )(x) . 
    \end{equation}
    Here, for $i=1, \dots ,n$, we write 
    \[ \mathcal{B}_i \coloneqq  \bigcup_{\substack{\tau \in S_{n} \\ \tau(n)=i}} \overline{C}_{e_1}([\varepsilon_{\tau (1)} \mid \dots \mid  \varepsilon_{\tau(n-1)} ]) \]
    and let $\mathcal{B}=\mathcal{B}_n$.
\end{lemma}

\begin{proof}
    Recall we write $V_\mathfrak{p} = V \oplus \langle \pi \rangle$ and, as in (\ref{e:unitgen}), choose the following generator for $H_n(V_{\mathfrak{p}},\mathbb{Z})$,
    \[ \eta_{\mathfrak{p}, V} =(-1)\sum_{\tau \in S_n} \sign(\tau)[\varepsilon_{\tau(1)} \mid \dots  \mid \varepsilon_{\tau(n)}]\otimes 1. \]
    Let $\sigma \in G$ and $\mathfrak{b}$ a fractional ideal coprime to $S$ and $\ell$ such that $\sigma = \sigma_\mathfrak{b}$. Recall,
    \[ u_3(V,\sigma) = (-1)^{n+1} (c_{\id} \cap (\omega_{\mathfrak{f}, \mathfrak{b}, \lambda,V}^\mathfrak{p} \cap \eta_{\mathfrak{p}, V} ))  . \]
    We calculate
    \[ \omega_{\mathfrak{f}, \mathfrak{b}, \lambda,V}^\mathfrak{p} \cap \eta_{\mathfrak{p}, V} =(-1)^{n+1} \sum_{i=1}^n \sum_{\substack{ \tau \in S_{n} \\ \tau (n)=i}} \sign(\tau) \omega_{\mathfrak{f}, \mathfrak{b}, \lambda,V}^\mathfrak{p}([\varepsilon_{\tau(1)} \mid \dots \mid \varepsilon_{\tau(n-1)} ]) \otimes [\varepsilon_i]. \]
    We recall the definition of $\omega_{\mathfrak{f}, \mathfrak{b}, \lambda,V}^\mathfrak{p}$ from $\S\ref{s:transfer}$. For $\tau \in S_{n}$ and a compact open $U \subseteq \mathcal{O}_\mathfrak{p}$, we have
\begin{equation}
    \sign(\tau) \omega_{\mathfrak{f}, \mathfrak{b}, \lambda,V}^\mathfrak{p}([\varepsilon_{\tau(1)} \mid \dots \mid  \varepsilon_{\tau(n-1)}]) = \zeta_{R, \lambda}(\mathfrak{b},\overline{C}_{e_1} ([\varepsilon_{\tau(1)} \mid \dots \mid  \varepsilon_{\tau(n-1)}]), U,0).
    \label{signcalc1}
\end{equation}
Returning to our main calculation, using (\ref{signcalc1}) we have
\begin{align*}
    c_\id \cap (\omega_{\mathfrak{f}, \mathfrak{b}, \lambda, V}^\mathfrak{p} \cap \eta_{\mathfrak{p}, V} ) &= (-1)^{n+1} \sum_{i=1}^n  \sum_{\substack{ \tau \in S_{n} \\ \tau (n)=i}} \int_{F_\mathfrak{p}} z_{\id}(\varepsilon_i)(x) \ d(\varepsilon_i \zeta_{R, \lambda}( \mathfrak{b},\overline{C}_{e_1}
([ \varepsilon_{\tau(1)} \mid \dots \mid \varepsilon_{\tau(n-1)} ]),x,0) ) \\
&= (-1)^{n+1}  \sum_{i=1}^n  \sum_{\substack{ \tau \in S_{n} \\ \tau (n)=i}} \int_{F_\mathfrak{p}} \varepsilon_i^{-1} z_{\id}(\varepsilon_i)(x) \ d( \zeta_{R, \lambda}( \mathfrak{b},\overline{C}_{e_1}
([ \varepsilon_{\tau(1)} \mid \dots \mid \varepsilon_{\tau(n-1)} ]),x,0) ) .
\end{align*}
We note that taking the cap product gives another factor of $(-1)$ which cancels the factor from before. One can easily compute, as is done in the proof of \cite[Proposition 4.6]{MR3968788}, that for $i=1,\dots ,n+r-1$, $i \neq n$
\begin{equation}
    \varepsilon_i^{-1} z_\id(\varepsilon_i)=\mathbbm{1}_{\pi \mathcal{O}_\mathfrak{p}} \cdot \varepsilon_i, 
    \label{calc2u3}
\end{equation}
and
\begin{equation}
    \pi^{-1}z_\id(\pi) = \mathbbm{1}_{\mathbb{O}}\cdot \id_{F_\mathfrak{p}^\ast} +\mathbbm{1}_{\mathcal{O}_\mathfrak{p}} \cdot \pi .
    \label{calc1u3}
\end{equation} 
Here we recall the $F_\fp^\ast$-action on $C_c(F_\fp^\ast , \mathbb{Z})$ from \S \ref{s:1cocycle}. Applying (\ref{calc2u3}) and (\ref{calc1u3}) and piecing together the appropriate Shintani sets we further deduce
\begin{multline}
    c_\id \cap (\omega_{\mathfrak{f}, \mathfrak{b}, \lambda,V}^\mathfrak{p} \cap \eta_{\mathfrak{p}, V} ) = (-1)^{n+1}\multint_{\mathbb{O}} x \ d(\zeta_{R,\lambda}(\mathfrak{b}, \mathcal{B},x,0)  )  \multint_{\mathcal{O}_\mathfrak{p}} \pi \ d(\zeta_{R,\lambda}(\mathfrak{b}, \mathcal{B},x,0)  ) \\  \prod_{i=1}^{n-1} \multint_{\pi\mathcal{O}_\mathfrak{p}}  \varepsilon_i \ d(\zeta_{R,\lambda}(\mathfrak{b}, \mathcal{B}_i,x,0)  ) .
    \label{calcstep1}
\end{multline}
It is clear that we can then write
\[ u_3(V, \sigma) = (-1)^{n+1}( c_\id \cap (\omega_{\mathfrak{f}, \mathfrak{b}, \lambda,V}^\mathfrak{p} \cap \eta_{\mathfrak{p}, V} ) ) =   \prod_{i=1}^{n-1} \varepsilon_i^{\zeta_{R,\lambda}(\mathfrak{b}, \mathcal{B}_i,\pi \mathcal{O}_\mathfrak{p},0)} \pi^{\zeta_{R,\lambda}(\mathfrak{b}, \mathcal{B},\mathcal{O}_\mathfrak{p},0) }   \multint_{\mathbb{O}} x \ d( \zeta_{R,\lambda}(\mathfrak{b}, \mathcal{B},x,0) )(x) . \]
    
\end{proof}

\section{Equality of $u_2$ and $u_3$} \label{s:u23}

In this section we prove the following theorem.

\begin{theorem} \label{u2equ3}
We have $u_2=u_3$.
\end{theorem}

In order to prove the above theorem we require the following lemma which records useful functorial properties of the cap-product.

\begin{lemma}\label{l: cap product properties}
    Let $G$ be a group, $H \subset G$ a subgroup and let $\iota : H \hookrightarrow G$ denote the inclusion map. Let $A,B,C$ be $G$-modules, $A^\prime,B^\prime,C^\prime$ be $H$-modules and suppose 
    \[ \epsilon : A \times B \rightarrow C, \quad (\text{resp.} \ \epsilon^\prime : A^\prime \times B^\prime \rightarrow C^\prime ) \]
    is a $G$-equivariant (resp. $H$-equrivariant) pairing inducing the cap-product pairing
    \begin{align*}
        \cap = \cap_\epsilon &: H^i (G,A) \times H_j(G,B) \rightarrow H_{j-i}(G,C)  \\
        (\text{resp. } \cap = \cap_{\epsilon^\prime} &: H^i (H,A^\prime) \times H_j(H,B^\prime) \rightarrow H_{j-i}(H,C^\prime) ) .
    \end{align*}
    Let $\alpha : A \rightarrow A^\prime$,  $\beta : B^\prime \rightarrow B$ and $\gamma : C^\prime \rightarrow C$ be $H$-equivariant maps such that 
    \[ \epsilon (a, \beta(b^\prime)) = \gamma (\epsilon^\prime (\alpha(a), b^\prime )) \quad \forall a \in A, b^\prime \in B^\prime . \]
    The maps $\iota, \alpha$ (resp. $\iota, \beta$) induce a homomorphism
    \begin{align*}
        r^G_H(\alpha) &= \alpha_\ast \circ \textrm{res}^G_H : H^i(G,A) \rightarrow H^i(H, A^\prime) \\
        (\text{resp. } c^G_H(\beta) &=  \textrm{cor}^G_H \circ \beta_\ast : H_j(G,B^\prime) \rightarrow H_j(H, B) ).
    \end{align*}
    Then for $a \in H^i(G,A)$ and $b^\prime \in H_j(H, B^\prime)$ the following formula holds
    \[ a \cap_\epsilon c^G_H (\beta) (b^\prime) = c^G_H(\gamma)( r_H^G(\alpha)(a) \cap_{\epsilon^\prime} b^\prime ).  \]
\end{lemma}

We modify the map $\Delta_\ast$ by omitting the place $\mathfrak{p}$ as well to obtain a $F^\ell_+$-equivariant map
\[ \Delta^\mathfrak{p} : \mathcal{C}_c^\diamond(R^\prime , \mathbb{Z}[G] )^{\mathfrak{p}, \ell , \infty}  \rightarrow C_c^0 (\mathbb{A}_F^{\mathfrak{p}, \ell , \infty}, \mathbb{Z}[G]) . \]
As before, the map $\Delta^\mathfrak{p}$ induces a homomorphism
\[ \Delta^\mathfrak{p}_\ast : H_n ( F_+^\ell , \mathcal{C}_c^\diamond(R^\prime , \mathbb{Z}[G] )^{\mathfrak{p}, \ell , \infty} ) \rightarrow H_n( F_+^\ell , C_c^0 (\mathbb{A}_F^{\mathfrak{p}, \ell , \infty}, \mathbb{Z}[G]) ) . \]
The natural pairing
\[ C_c(F_\mathfrak{p}, F_\mathfrak{p}^\ast) \times C_c^0 (\mathbb{A}_F^{\mathfrak{p}, \ell , \infty}, \mathbb{Z}[G]) \rightarrow C_c^0 (\mathbb{A}_F^{ \ell , \infty}, F_\mathfrak{p}^\ast \otimes \mathbb{Z}[G]) \]
induces a cap-product pairing
\[ H^1( F^\ell_+ , C_c(F_\mathfrak{p}, F_\mathfrak{p}^\ast) ) \times H_n( F^\ell_+ , C_c^0 (\mathbb{A}_F^{\mathfrak{p}, \ell , \infty}, \mathbb{Z}[G]) ) \rightarrow H_{n-1} ( F^\ell_+ , C_c^0 (\mathbb{A}_F^{ \ell , \infty}, F_\mathfrak{p}^\ast \otimes \mathbb{Z}[G]) ) . \]
It is straightforward to see that we then have the equality
\begin{equation}\label{e: new defn for u2}
    u_2 =  \Eis_{F,\lambda}^0 \cap ( c_\id \cap \Delta^\mathfrak{p}_\ast (\rho_{H/F}) ) .
\end{equation}
We are now ready to prove the main theorem of this section.

\begin{proof}[Proof of Theorem \ref{u2equ3}]
    Let $\mathcal{F} \subseteq ( \mathbb{A}_F^{\mathfrak{p}, \ell, \infty} )^\ast / U_\mathfrak{f}^{\mathfrak{p}, \ell, \infty} $ be a fundamental domain for the action of \\ $F_+^\ell / E_+(\mathfrak{f})_\mathfrak{p}$. Consider the $E_+(\mathfrak{f})_\mathfrak{p}$-equivariant map 
    \[ j : \mathbb{Z} \rightarrow C_c( ( \mathbb{A}_F^{\mathfrak{p}, \ell, \infty} )^\ast / U_\mathfrak{f}^{\mathfrak{p}, \ell, \infty}, \mathbb{Z} ), \quad 1 \mapsto \mathbbm{1}_{\mathcal{F}} . \]
    Let $ \eta_{\mathfrak{p},E_+(\mathfrak{f}) }   \in H_{n}(E_+(\mathfrak{f})_{\mathfrak{p}}, \mathbb{Z})$ be the generator defined in (\ref{e:unitgen}) and let $\vartheta_\mathfrak{f}^\mathfrak{p}$ be the image of the class $\eta_{\mathfrak{p},E_+(\mathfrak{f}) }$ under the map
    \[ c_{E_+(\mathfrak{f})_\mathfrak{p}}^{F_+^\ell} : H_n(E_+(\mathfrak{f})_\mathfrak{p} , \mathbb{Z}) \rightarrow H_n(F_+^\ell , C_c( ( \mathbb{A}_F^{\mathfrak{p}, \ell, \infty} )^\ast / U_\mathfrak{f}^{\mathfrak{p}, \ell, \infty}, \mathbb{Z} )) .  \]
    Here this map is induced by the map $j$, defined above. Let $\pi : C_c( ( \mathbb{A}_F^{\mathfrak{p}, \ell, \infty} )^\ast / U_\mathfrak{f}^{\mathfrak{p}, \ell, \infty}, \mathbb{Z} ) \rightarrow C_c( ( \mathbb{A}_F^{\mathfrak{p}, \ell, \infty} )^\ast / U^{\mathfrak{p}, \ell, \infty}, \mathbb{Z} ) $ be the natural projection and write 
    \[ \iota : C_c( ( \mathbb{A}_F^{\mathfrak{p}, \ell, \infty} )^\ast / U^{\mathfrak{p}, \ell, \infty}, \mathbb{Z} ) \rightarrow C_c( ( \mathbb{A}_F^{\mathfrak{p}, \ell, \infty} )^\ast / U_\mathfrak{f}^{\mathfrak{p}, \ell, \infty}, \mathbb{Z} ) , \quad \varphi \mapsto \varphi \circ \pi \]
    for the induced map. Using the fact that $\textrm{cor}_{E_+(\mathfrak{f})_\mathfrak{p}}^{E_{\mathfrak{p}, +}}(\eta_\mathfrak{p}) = \eta_{\mathfrak{p}, E_+(\mathfrak{f})}$ one can observe that 
    \begin{equation}\label{e: iota map}
        \iota_\ast ( \vartheta^\mathfrak{p} ) = \vartheta_\mathfrak{f}^\mathfrak{p},
    \end{equation}
    (for more details we refer to the proof of \cite[Lemma 5.1]{AdelicEisensteinClasses}). Note that the reciprocity map $\rec_{H/F}^{\mathfrak{p}, \ell, \infty} $ factors through $(\mathbb{A}_F^{\mathfrak{p}, \ell, \infty})^\ast/ U_\mathfrak{f}^{\mathfrak{p}, \ell, \infty}$, to distinguish it from $\rec_{H/F}^{\mathfrak{p}, \ell, \infty} $ we write
    \[ \rec_{H/F, \mathfrak{f}}^{\mathfrak{p}, \ell, \infty} : (\mathbb{A}_F^{\mathfrak{p}, \ell, \infty})^\ast/ U_\mathfrak{f}^{\mathfrak{p}, \ell, \infty} \rightarrow G \hookrightarrow \mathbb{Z}[G]^\ast .  \]
    As before we can view $\rec_{H/F, \mathfrak{f}}^{\mathfrak{p}, \ell, \infty}$ as an element of $H^0( F_+^\ell , C( (\mathbb{A}_F^{\mathfrak{p}, \ell, \infty})^\ast/ U_\mathfrak{f}^{\mathfrak{p}, \ell, \infty} , \mathbb{Z}[G] ) )$. By (\ref{e: iota map}) we have the equality
    \begin{equation}\label{e: delta equality}
        \Delta_\ast^\mathfrak{p}(\rho_{H/F}) = \Delta_\ast^\mathfrak{p} ( \rec_{H/F, \mathfrak{f}}^{\mathfrak{p}, \ell, \infty} \cap \vartheta_\mathfrak{f}^\mathfrak{p} ) .
    \end{equation}
    Here the cap-product is induced by the pairing
    \begin{align*}
        C( (\mathbb{A}_F^{\mathfrak{p}, \ell, \infty})^\ast/ U_\mathfrak{f}^{\mathfrak{p}, \ell, \infty} , \mathbb{Z}[G] ) \times  C_c( (\mathbb{A}_F^{\mathfrak{p}, \ell, \infty})^\ast/ U_\mathfrak{f}^{\mathfrak{p}, \ell, \infty} , \mathbb{Z} ) &\rightarrow C( (\mathbb{A}_F^{\mathfrak{p}, \ell, \infty})^\ast/ U_\mathfrak{f}^{\mathfrak{p}, \ell, \infty} , \mathbb{Z}[G] ), \\ (\phi, \psi) &\mapsto \phi \odot \psi .
    \end{align*}
    For any $b \in ( \mathbb{A}_F^{\mathfrak{p}, \ell, \infty} )^\ast$ we define maps
    \begin{align*}
        j_b &: \mathbb{Z} \rightarrow  C_c( (\mathbb{A}_F^{\mathfrak{p}, \ell, \infty})^\ast/ U_\mathfrak{f}^{\mathfrak{p}, \ell, \infty} , \mathbb{Z} ) , \quad 1 \mapsto \mathbbm{1}_{bU_\mathfrak{f}^{\mathfrak{p}, \ell, \infty}} \\
        \textrm{ev}_b &: C( (\mathbb{A}_F^{\mathfrak{p}, \ell, \infty})^\ast/ U_\mathfrak{f}^{\mathfrak{p}, \ell, \infty} , \mathbb{Z}[G] ) \rightarrow \mathbb{Z}[G], \quad \varphi \mapsto \varphi(bU_\mathfrak{f}^{\mathfrak{p}, \ell, \infty}). 
    \end{align*}
    Note that for every $m \in \mathbb{Z}$ and $\varphi \in C( (\mathbb{A}_F^{\mathfrak{p}, \ell, \infty})^\ast/ U_\mathfrak{f}^{\mathfrak{p}, \ell, \infty} , \mathbb{Z}[G] ) $ we have
    \begin{equation}\label{e: j equality}
        j_b(m) \odot \varphi = j_b ( m \cdot \textrm{ev}_b (\varphi) ) .
    \end{equation}
    Let $\vartheta_b^\mathfrak{p} \in H_n(F_+^\ell , C_c( (\mathbb{A}_F^{\mathfrak{p}, \ell, \infty})^\ast/ U_\mathfrak{f}^{\mathfrak{p}, \ell, \infty} , \mathbb{Z} ) $ be the image of $\eta_{\mathfrak{p}, E_+(\mathfrak{f})}$ under the map 
    \[ c_{E_+(\mathfrak{f})_\mathfrak{p}}^{F_+^\ell}(j_b) : H_n( E_+(\mathfrak{f} )_\mathfrak{p} , \mathbb{Z}) \rightarrow H_n( F_+^\ell , C_c( (\mathbb{A}_F^{\mathfrak{p}, \ell, \infty})^\ast/ U_\mathfrak{f}^{\mathfrak{p}, \ell, \infty} , \mathbb{Z} ) ) . \]
    For the fundamental domain $\mathcal{F} = \{ 
b_1 U_\mathfrak{f}^{\mathfrak{p}, \ell, \infty}, \dots , b_s U_\mathfrak{f}^{\mathfrak{p}, \ell, \infty} \}$ we choose $b_1, \dots b_s \in (\mathbb{A}_F^{\mathfrak{p}, \ell, \infty})^\ast$ such that the $b_i$ are coprime to $\mathfrak{f}$. We then write $\mathfrak{b}_1, \dots \mathfrak{b}_s$ for the corresponding fractional ideals of $\mathcal{O}_F$ and note that, by our choice of $b_1, \dots , b_s$, the ideals $\mathfrak{b}_1, \dots , \mathfrak{b}_s$ are coprime to $\mathfrak{f} \cdot \ell \cdot \mathfrak{p}$. Since
\[ (\mathbb{A}_F^{\mathfrak{p}, \ell, \infty})^\ast/ F_+^\ell U_\mathfrak{f}^{\mathfrak{p}, \ell, \infty} \cong G_\mathfrak{f}/ \langle \mathfrak{p} \rangle, \]
the collection of fractional ideals $\mathfrak{b}_1, \dots, \mathfrak{b}_s $ is a system of representatives of $G_\mathfrak{f}/ \langle \mathfrak{p} \rangle$. It follows that
\[ \vartheta_\mathfrak{f}^\mathfrak{p} = \sum_{i=1}^s \vartheta_{b_i}^\mathfrak{p} . \]
Hence, by (\ref{e: delta equality}) we can calculate
\begin{equation}\label{e: delta as a sum}
    \Delta_\ast^\mathfrak{p}( \rho_{H/F} ) = \sum_{i=1}^s \Delta_\ast^\mathfrak{p}( \rec_{H/F, \mathfrak{f}}^{\mathfrak{p}, \ell, \infty} \cap \vartheta_{b_i}^\mathfrak{p} ) . 
\end{equation}
For every $i \in \{ 1, \dots , s \}$ we have, by Lemma \ref{l: cap product properties} and (\ref{e: j equality}), 
\begin{equation}\label{e: rec with c map}
    \rec_{H/F, \mathfrak{f}}^{\mathfrak{p}, \ell, \infty} \cap \vartheta_{b_i}^\mathfrak{p} = c_{E_+(\mathfrak{f})_\mathfrak{p}}^{F_+^\ell}(j_{b_i}) ( \rec_{H/F}(\mathfrak{b}_i ) \otimes \eta_{\mathfrak{p}, E_+(\mathfrak{f})} ) = c_{E_+(\mathfrak{f})_\mathfrak{p}}^{F_+^\ell}(j_{b_i})(  \eta_{\mathfrak{p}, E_+(\mathfrak{f})} ) \otimes \rec_{H/F}(\mathfrak{b}_i) .
\end{equation}
Under the composition $\Delta^\mathfrak{p} \circ j_{b_i} : \mathbb{Z} \rightarrow C_c^0( \mathbb{A}_F^{\mathfrak{p}, \ell, \infty} , \mathbb{Z} ) $ we have that $1$ is mapped to the characteristic function of $\widehat{\mathfrak{b}_i}^{\mathfrak{p}, \ell} \coloneqq \mathfrak{b}_i \prod^\prime_{v \nmid \mathfrak{p}, \ell, \infty} \mathcal{O}_v $. We define the map
\[ \delta_{\mathfrak{b}_i} : \Hom( C_c^0(\mathbb{A}_F^{\ell , \infty} , \mathbb{Z} ), \mathbb{Z} ) \rightarrow \Hom (C_c(F_\mathfrak{p}, \mathbb{Z}), \mathbb{Z} )  \]
by $ \delta_{\mathfrak{b}_i} (\mu) (f) = \mu(f \otimes \mathbbm{1}_{\widehat{\mathfrak{b}_i}^{\mathfrak{p}, \ell}})$ for $\mu \in \Hom( C_c^0(\mathbb{A}_F^{\ell , \infty} , \mathbb{Z} ), \mathbb{Z} ) $ and $f \in C_c(F_\mathfrak{p}, \mathbb{Z})$. It then follows from the definitions of $\Eis_{F,\lambda}^0$ and $\omega_{\mathfrak{f}, \mathfrak{b}_i, \lambda}$ that 
\begin{equation}\label{e: reln between Eis and w}
    r_{E_+(\mathfrak{f})_\mathfrak{p}}^{F_+^\ell}(\delta_{\mathfrak{b}_i})(\Eis_{F,\lambda}^0) = \omega_{\mathfrak{f}, \mathfrak{b}_i, \lambda} . 
\end{equation}
Here 
\[ r_{E_+(\mathfrak{f})_\mathfrak{p}}^{F_+^{\ell}}(\delta_{\mathfrak{b}_i}) : H^{n-1}( F_+^\ell , \Hom( C_c^0(\mathbb{A}_F^{\ell , \infty} , \mathbb{Z} ), \mathbb{Z} )  ) \rightarrow H^{n-1}( E_+(\mathfrak{f})_\mathfrak{p} , \Hom (C_c(F_\mathfrak{p}, \mathbb{Z}), \mathbb{Z} ) ) \]
is the induced map, as defined in Lemma \ref{l: cap product properties}. We calculate, using standard properties of the cap-product and (\ref{e: rec with c map}), that
\begin{align*}
    \Eis_{F, \lambda}^0 \cap ( c_\id \cap \Delta_\ast^\mathfrak{p} ( \rec_{H/F, \mathfrak{f}}^{\mathfrak{p}, \ell, \infty} \cap \vartheta_{b_i}^\mathfrak{p} ) ) &= (-1)^{n-1} c_\id \cap ( \Eis_{F, \lambda}^0 \cap \Delta_\ast^\mathfrak{p} ( \rec_{H/F, \mathfrak{f}}^{\mathfrak{p}, \ell, \infty} \cap \vartheta_{b_i}^\mathfrak{p} ) ) \\
    &= (-1)^{n-1} c_\id \cap ( \Eis_{F, \lambda}^0 \cap \Delta_\ast^\mathfrak{p} ( c_{E_+(\mathfrak{f})_\mathfrak{p}}^{F_+^\ell}(j_{b_i})(  \eta_{\mathfrak{p}, E_+(\mathfrak{f})} ) \otimes \rec_{H/F}(\mathfrak{b}_i) )) \\
    &= (-1)^{n-1} c_\id \cap ( \Eis_{F, \lambda}^0 \cap \Delta_\ast^\mathfrak{p} ( c_{E_+(\mathfrak{f})_\mathfrak{p}}^{F_+^\ell}(j_{b_i})(  \eta_{\mathfrak{p}, E_+(\mathfrak{f})} ))) \otimes \rec_{H/F}(\mathfrak{b}_i)  \\
    &= (-1)^{n-1} c_\id \cap ( \Eis_{F, \lambda}^0 \cap   c_{E_+(\mathfrak{f})_\mathfrak{p}}^{F_+^\ell}(\Delta^\mathfrak{p} \circ j_{b_i})(  \eta_{\mathfrak{p}, E_+(\mathfrak{f})} )) \otimes \rec_{H/F}(\mathfrak{b}_i) .
\end{align*}
Applying Lemma \ref{l: cap product properties} and the equality (\ref{e: reln between Eis and w}) to the above calculation then yields,
\begin{align}\notag
    \Eis_{F, \lambda}^0 \cap ( c_\id \cap \Delta_\ast^\mathfrak{p} ( \rec_{H/F, \mathfrak{f}}^{\mathfrak{p}, \ell, \infty} \cap \vartheta_{b_i}^\mathfrak{p} ) )
    &= (-1)^{n-1} c_\id \cap ( \omega_{\mathfrak{f}, \mathfrak{b}_i, \lambda} \cap     \eta_{\mathfrak{p}, E_+(\mathfrak{f})} ) \otimes \rec_{H/F}(\mathfrak{b}_i) \\ \label{e: one term of u2 sum}
    &= (-1)^{n-1} c_\id \cap (( \omega_{\mathfrak{f}, \mathfrak{b}_i, \lambda} \otimes \rec_{H/F}(\mathfrak{b}_i)) \cap     \eta_{\mathfrak{p}, E_+(\mathfrak{f})} ) .
\end{align}
Recalling the alternative definition of $u_2$ from (\ref{e: new defn for u2}) and the equality (\ref{e: delta as a sum}) we have 
\[ u_2 =  \Eis_{F,\lambda}^0 \cap ( c_\id \cap \Delta^\mathfrak{p}_\ast (\rho_{H/F}) ) 
    =  \sum_{i=1}^s (-1)^{n-1} ( \Eis_{F, \lambda}^0 \cap ( c_\id \cap \Delta_\ast^\mathfrak{p} ( \rec_{H/F, \mathfrak{f}}^{\mathfrak{p}, \ell, \infty} \cap \vartheta_{b_i}^\mathfrak{p} )) ) . \]
Applying the calculation in (\ref{e: one term of u2 sum}) we can then observe 
\begin{align*}
    u_2 &= (-1)^{n-1} \sum_{i=1}^s   c_\id \cap (( \omega_{\mathfrak{f}, \mathfrak{b}_i, \lambda} \otimes \rec_{H/F}(\mathfrak{b}_i)) \cap     \eta_{\mathfrak{p}, E_+(\mathfrak{f})} ) \\
    &= (-1)^{n-1} (c_\id \cap (( \sum_{i=1}^s \omega_{\mathfrak{f}, \mathfrak{b}_i, \lambda} \otimes \rec_{H/F}(\mathfrak{b}_i)) \cap     \eta_{\mathfrak{p}, E_+(\mathfrak{f})} ))\\
    &= (-1)^{n-1} ( c_\id \cap (  \omega_{\mathfrak{f},  \lambda}  \cap     \eta_{\mathfrak{p}, E_+(\mathfrak{f})} )) \\
    &=  u_3 .
\end{align*}
Here the third equality is due to the definition of $\omega_{\mathfrak{f},  \lambda}$ from (\ref{e: defn of w as a sum}). The final equality is simply the definition of $u_3$ and the fact that $(-1)^{n-1}=(-1)^{n+1}$.

\end{proof}

\section{Equality of $u_1$ and $u_3$} \label{s:u12}

In this section we prove the following theorem.
\begin{theorem} \label{mainnewthm}
We have $u_1=u_3$.
\end{theorem}
We show that for each $\sigma \in G$ we have $u_1(\sigma)=u_3(\sigma)$. This is done by using a strong enough compatibility property which forces the formulas to be equal. A special argument will be required in the case that $R$ contains no finite places, i.e., $R = R_\infty$.

We are given a CM abelian extension $H/F$ of conductor $\mathfrak{f}$ such that $\mathfrak{p}$ splits completely in $H$.  Let $\mathfrak{f}'$ be an auxiliary ideal of $\mathcal{O}_F$ that is divisible only by primes dividing $\mathfrak{f}$. Let $H^\prime$ be another finite abelian CM extension of $F$ in which $\mathfrak{p}$ splits completely, such that the conductor of $H'/F$ divides $\mathfrak{f}\mathfrak{f}'$. In particular, the extension $H^\prime/F$ is unramified outside $R$.

Let $\sigma \in G$. Write $u_1(\sigma, H)$ and $u_3(\sigma, H)$ for $\sigma$ components of the formulas $u_1$ and $u_3$, for the extension $H/F$ and Galois group element $\sigma$. We show that, for $i=1,3$,
\begin{equation}
    u_i(\sigma, H) = \prod_{\substack{\tau \in G^\prime \\ \tau \mid_H = \sigma } } u_i(\tau, H^\prime)  .
    \label{normcompbasic}
\end{equation}
We refer to (\ref{normcompbasic}) as norm compatibility. 

\begin{proposition} \label{uptoerrorterm}
We have 
\[ u_1(\sigma, H) \equiv u_3(\sigma, H) \pmod{E_+(\mathfrak{f})} \]
\end{proposition}

\begin{proof}
Let $V$ be a finite index subgroup of $E_+(\mathfrak{f})$ satisfying the conditions given in the statement of Proposition \ref{firstchangeofdom}. Furthermore, we choose $V$ such that if $V= \langle \varepsilon_1 , \dots , \varepsilon_{n-1} \rangle$ then the $\varepsilon_i$ along with $\pi$ satisfy Lemma \ref{lemmatomakembig}. 
We recall from \S\ref{s:eeu3p} the explicit description of $u_3(V, \sigma)$,
\[ u_3(V, \sigma) =  c_\id \cap (\omega_{\mathfrak{f}, \mathfrak{b}, \lambda,V}^\mathfrak{p} \cap \eta_{\mathfrak{p}, V} )  =   \prod_{i=1}^{n-1} \varepsilon_i^{\zeta_{R,\lambda}(\mathfrak{b}, \mathcal{B}_i,\pi \mathcal{O}_\mathfrak{p},0)} \pi^{\zeta_{R,\lambda}(\mathfrak{b}, \mathcal{B},\mathcal{O}_\mathfrak{p},0) }   \multint_{\mathbb{O}} x \ d( \zeta_{R,\lambda}(\mathfrak{b}, \mathcal{B},x,0) )(x) . \] 
We have defined
\[ u_1(V, \mathcal{B}, \sigma) =  \prod_{\epsilon \in V} \epsilon^{\zeta_{R,\lambda}(\mathfrak{b}, \epsilon \mathcal{B} \cap \pi^{-1} \mathcal{B}, \mathcal{O}_\mathfrak{p},0) } \pi^{\zeta_{R,\lambda}( \mathfrak{b}, \mathcal{B},\mathcal{O}_\mathfrak{p},0)} \multint_\mathbb{O} x \ d \nu_\lambda (\mathfrak{b}, \mathcal{B},x), \]
where 
\[ \mathcal{B} = \bigcup_{\tau \in S_{n-1}} \overline{C}_{e_1}([\varepsilon_{\tau(1)} \mid \dots \mid \varepsilon_{\tau(n-1)} ]) . \]
Thus, $ u_1(V, \mathcal{B}, \sigma) \equiv u_3(V, \sigma) \pmod{E_+(\mathfrak{f})} $ and hence there exists $\alpha \in F_\mathfrak{p}^\ast$ such that 
\begin{equation} \label{e:alphav}
\alpha u_1(\sigma, H) = u_3( \sigma, H) \quad \text{and} \quad \alpha^{[E_+:V]} \in E_+(\mathfrak{f}). \end{equation}
Note that here we are also using Proposition \ref{canchangedom}. By Lemma \ref{l:coprimechoices}, we can choose $W$ to be a finite index subgroup of $E_+(\mathfrak{f})$ satisfying the same conditions as $V$ but with $[E_+(\mathfrak{f}):W]$ coprime to $[E_+(\mathfrak{f}):V]$. Thus we also have $\alpha^{[E_+(\mathfrak{f}):W]} \in E_+(\mathfrak{f})$, which combines with (\ref{e:alphav}) to yield $\alpha \in E_+(\mathfrak{f})$ as desired.
\end{proof}

Assuming that (\ref{normcompbasic}) holds we can prove the following. 

\begin{proposition}\label{p: u1=u2 if f>1}
    Suppose that (\ref{normcompbasic}) holds and that $R \neq R_\infty$. Then,
    \[ u_1 = u_3 . \]
\end{proposition}

\begin{proof}
From Proposition \ref{uptoerrorterm} we have that for each $\tau \in G^\prime$,
\[ u_1(\tau, H^\prime) \equiv u_3(\tau, H^\prime) \pmod{E_+(\mathfrak{ff}^\prime)}. \]
Our assumption that (\ref{normcompbasic}) holds then gives that for each $\sigma \in G$,
\[ u_1(\sigma, H) \equiv u_3(\sigma, H) \pmod{E_+(\mathfrak{ff}^\prime)}. \]
Since $R \neq R_\infty$, we have
\[ \bigcap_{\mathfrak{f}^\prime} E_+(\mathfrak{ff}^\prime) = \{ 1 \} . \]
Here the intersection is taken over all possible ideals $\mathfrak{f}^\prime$ divisible only by primes dividing $\mathfrak{f}$. Thus we have
\[ u_1(\sigma, H) = u_3(\sigma, H). \]

\end{proof}

\begin{remark}
If $R=R_\infty$ then $\mathfrak{f}= \mathfrak{ff}^\prime =1$ for all possible extensions. Hence, the proof of Proposition~\ref{p: u1=u2 if f>1} does not apply.
\end{remark}

To handle the case $R=R_\infty$ we extend the definition of $u_1$ to work with the trivial extension. For a Shintani set $\mathcal{D}$ and compact open $U \subseteq \mathcal{O}_\mathfrak{p}$, we define
\[ \nu_\lambda(\mathcal{D}, U) =   \zeta_{R,\lambda}( \mathcal{O}_F , \mathcal{D}, U , 0) .  \]
It is clear that 
\begin{equation}\label{e: norm for measure in f=1}
    \nu_\lambda(\mathcal{D}, U) = \sum_{\sigma_\mathfrak{b} \in G} \nu_\lambda(\mathfrak{b}, \mathcal{D}, U) .
\end{equation}
We then define, for a Shintani domain $\mathcal{D}$,
\begin{equation} \label{e:u1fdef}
u_1(F)= \left( \prod_{\epsilon \in E_+} \epsilon^{\nu_\lambda( \epsilon \mathcal{D} \cap \pi^{-1} \mathcal{D}, \mathcal{O}_\mathfrak{p})} \right) \pi^{\nu_\lambda(\mathcal{D},  \mathcal{O}_\mathfrak{p})} \multint_\mathbb{O} x \ d \nu_\lambda ( \mathcal{D},x) . \end{equation}
By (\ref{e: norm for measure in f=1}) and since $\mathfrak{f}=1$ we have
\[ u_1(F)= \prod_{\sigma \in G} u_1(\sigma , H). \]
\begin{lemma}
    We have
    \[ u_1(F)=1 . \]
\end{lemma}

\begin{proof}
    Since $\mathcal{D}$ is a Shintani domain we have 
    \begin{equation}\label{e: trivial measure}
        \nu_\lambda(\mathcal{D}, \mathcal{O}_\mathfrak{p}) = \zeta_{R,\lambda}(F/F, \mathcal{O}_F , 0) = 0 . 
    \end{equation}
    Therefore the $\pi$-power term in (\ref{e:u1fdef}) vanishes.
Next, we write
    \begin{equation}\label{e: fraction of int}
        \multint_\mathbb{O} x \ d \nu_\lambda ( \mathcal{D},x) = \frac{\multint_{\mathcal{O}_\mathfrak{p}} x \ d \nu_\lambda ( \mathcal{D},x) }{ \multint_{\pi\mathcal{O}_\mathfrak{p}} x \ d \nu_\lambda ( \mathcal{D},x) } .
    \end{equation}
    By Lemma \ref{changeofvariable} we calculate
    \begin{equation} \label{e:xnupid}
\multint_{\pi\mathcal{O}_\mathfrak{p}} x \ d \nu_\lambda ( \mathcal{D},x) = \pi^{\nu_\lambda(\pi \mathcal{D}, \pi \mathcal{O}_\mathfrak{p})} \multint_{\mathcal{O}_\mathfrak{p}} x \ d \nu_\lambda ( \pi^{-1}\mathcal{D},x)
    = \multint_{\mathcal{O}_\mathfrak{p}} x \ d \nu_\lambda ( \pi^{-1}\mathcal{D},x) 
    \end{equation}
since  $\nu_\lambda(\pi \mathcal{D}, \pi \mathcal{O}_\mathfrak{p}) =0$ as in (\ref{e: trivial measure}). Since $\mathcal{D}$ is a Shintani domain we can write 
    \[ \pi^{-1} \mathcal{D} = \bigcup_{\epsilon \in E_+} (\epsilon \mathcal{D} \cap \pi^{-1}\mathcal{D}) . \]
    We then have
    \begin{align}
    \begin{split}
        \multint_{\mathcal{O}_\mathfrak{p}} x \ d \nu_\lambda ( \pi^{-1}\mathcal{D},x) &= \prod_{\epsilon \in E_+} \left( \multint_{\mathcal{O}_\mathfrak{p}} x \ d \nu_\lambda ( \epsilon \mathcal{D} \cap \pi^{-1}\mathcal{D},x)  \right) \\ \label{e: calc of int}
        &= \left( \prod_{\epsilon \in E_+} \epsilon^{\nu_\lambda(\epsilon \mathcal{D} \cap \pi^{-1} \mathcal{D} , \mathcal{O}_\mathfrak{p} ) } \right) \multint_{\mathcal{O}_\mathfrak{p}} x \ d \nu_\lambda (  \mathcal{D} ,x) .
    \end{split}
    \end{align}
    Combining (\ref{e: fraction of int}), (\ref{e:xnupid}), and (\ref{e: calc of int}) yields
    \[  \multint_\mathbb{O} x \ d \nu_\lambda ( \mathcal{D},x) =  
    \left( \prod_{\epsilon \in E_+} \epsilon^{\nu_\lambda(\epsilon \mathcal{D} \cap \pi^{-1} \mathcal{D} , \mathcal{O}_\mathfrak{p} ) } \right)^{-1}
    \]
Applying the definition of $u_1(F)$ yields the desired result.
\end{proof}

\begin{proposition}\label{p: u1=u2 if f=1}
    Suppose that (\ref{normcompbasic}) holds and that $R = R_\infty$. Then,
    \[ u_1 = u_3 . \]
\end{proposition}

\begin{proof}
    Let $\sigma \in G$. By Proposition \ref{uptoerrorterm} there exists $\varepsilon(\sigma) \in E_+$ such that 
    \[ u_1(\sigma)= \varepsilon(\sigma) u_3(\sigma) . \]
    Let $\mathfrak{r}$ be a prime of $F$. From the equation
    \[ \zeta_{R \cup \{ \mathfrak{r} \} }( \mathfrak{b}, \mathcal{D}, U, s) = \zeta_{R  }( \mathfrak{b}, \mathcal{D}, U, s) - \N \mathfrak{r}^{-s} \zeta_{R  }( \mathfrak{br}^{-1}, \mathcal{D}, U, s),  \]
    it follows that
    \[ u_1(S \cup \{ \mathfrak{r} \} , \sigma) = u_1(S ,  \sigma) u_1(S , \sigma_\mathfrak{r}^{-1} \sigma)^{-1}. \] 
    Proposition \ref{prop6.3} \textit{c)} gives the same result for $u_2$. Applying Theorem \ref{u2equ3} we have the same result for $u_3$, therefore
    \begin{align}
        u_3(S ,  \sigma) u_3(S , \sigma_\mathfrak{r}^{-1} \sigma)^{-1} &= u_3(S \cup \{ \mathfrak{r} \} , \sigma) \notag \\ \label{e: calc for f=1, 1}
        &= u_1(S \cup \{ \mathfrak{r} \} , \sigma) \\
        &= u_1(S ,  \sigma) u_1(S , \sigma_\mathfrak{r}^{-1} \sigma)^{-1} \notag \\
        &= \varepsilon(\sigma)\varepsilon(\sigma \sigma_\mathfrak{r}^{-1})^{-1} u_3(S ,  \sigma) u_3(S , \sigma_\mathfrak{r}^{-1} \sigma)^{-1} . \notag
    \end{align}
    Here, (\ref{e: calc for f=1, 1}) is given by Proposition \ref{p: u1=u2 if f>1}, which can be applied since we have added $\mathfrak{r}$ to the set $R$. It follows that $\varepsilon(\sigma)= \varepsilon(\sigma \sigma_\mathfrak{r}^{-1})$. Repeating this for all such $\mathfrak{r}$ we see that $\varepsilon(\sigma)$ is independent of $\sigma \in G$. Write $\varepsilon = \varepsilon(\sigma)$. Then 
    \[ 1= u_1(F) = \prod_{\sigma \in G} u_1(\sigma, H) = \varepsilon^{\mid G \mid} \prod_{\sigma \in G} u_3(\sigma, H) = \varepsilon^{\mid G \mid} . \]
    The last equality follows since $\prod_{\sigma \in G}u_2(\sigma , H) = 1$ by \ref{prop6.3} \textit{b),d)} and Theorem \ref{u2equ3}. Since $\varepsilon \in E_+$, it follows that $\varepsilon = 1$. This gives the desired result.
\end{proof} 

Theorem \ref{mainnewthm}, under the assumption that (\ref{normcompbasic}) holds, then follows from the combination of Proposition \ref{p: u1=u2 if f>1} and Proposition \ref{p: u1=u2 if f=1}. In the next section we prove the norm compatibility property (\ref{normcompbasic}) for $u_1$ and $u_2$.

\section{Norm compatibility relations}

In this section we prove norm compatibility properties for $u_1$ and $u_3$.

\subsection{Norm compatibility for $u_1$}

In this section we again allow any choice of appropriate $T$.

The reciprocity map identifies $\Gal(H^\prime/ H)$ with 
\begin{equation}\label{e:setofcong}
    \{ \beta \in (\mathcal{O}_F/\mathfrak{ff}^\prime)^\ast \mid \beta \equiv 1 \pmod{\mathfrak{f}} \} / E_+(\mathfrak{f})_\mathfrak{p}.
\end{equation}
We let $\mathcal{D}_\mathfrak{f}$ be a Shintani domain for $E_+(\mathfrak{f})$ and define 
\[ \mathcal{D}_{\mathfrak{ff}^\prime} = \bigcup_{\gamma \in E_+(\mathfrak{f})/ E_+(\mathfrak{ff}^\prime )} \gamma \mathcal{D}_\mathfrak{f} , \]
where the union is over a set of representatives $\{ \gamma \}$ for $E_+(\mathfrak{ff}^\prime)$ in $E_+(\mathfrak{f})$. Let $e^\prime$ be the order of $\mathfrak{p}$ in $G_{\mathfrak{ff}^\prime}$, and suppose that $\mathfrak{p}^{e^\prime}=(\pi^\prime)$ with $\pi^\prime$ totally positive and $\pi^\prime \equiv 1 \pmod{\mathfrak{ff}^\prime}$. We can choose $\pi^\prime$ such that $\pi^\prime = \pi^\alpha$ for some $\alpha \geq 1$. We then define $\mathbb{O}^\prime = \mathcal{O}_\mathfrak{p} - \pi^\prime \mathcal{O}_\mathfrak{p}$.

Let $B$ denote a set of totally positive elements of $\mathcal{O}_F$ that are relatively prime to $S$ and $\overline{T}$ and whose images in $(\mathcal{O}_F/ \mathfrak{ff}^\prime)^\ast$ are a set of distinct representatives for (\ref{e:setofcong}).

The following theorem is stated without proof by the first author in \cite[Theorem 7.1]{MR2420508}. For completeness we include a proof of this result here. 

\begin{theorem}[Theorem 7.1, \cite{MR2420508}] \label{normforu1}
We have
\[ u_T (\mathfrak{b} , \mathcal{D}_\mathfrak{f} ) = \prod_{\beta \in B} u_T (\mathfrak{b}(\beta) , \beta^{-1} \mathcal{D}_{\mathfrak{ff}^\prime} ) . \]
\end{theorem}

The key to the proof of Theorem \ref{normforu1} is to use translation properties of Shintani sets. For a subset $A$ of equivalence classes of (\ref{e:setofcong}), let $\nu^A_T (\mathfrak{b}, \mathcal{D}, U) = \zeta^A_{R,T}(\mathfrak{b}, \mathcal{D}, U, 0)$, where $\zeta^A_R$ is the zeta function
\[ \zeta^A_R (\mathfrak{b}, \mathcal{D}, U, s) = \N \mathfrak{b}^{-s} \sum_{\substack{ \alpha \in \mathfrak{b}^{-1}\cap \mathcal{D}, \ \alpha \in U \\ \alpha \in A , \ (\alpha, R)=1 }} \N \alpha^{-s}. \]
This definition extends to $\zeta^A_{R,T}$ as in (\ref{pzfT}). Throughout this section we will use the following simple equality:
\[ \nu^{ \{ \pi^{-1} \} }_T (\mathfrak{b}, \mathcal{D}, U) = \nu^{ \{ 1 \} }_T (\mathfrak{b}, \pi \mathcal{D}, \pi U).  \]
This follows from Lemma \ref{changeofvariable}. Recall the following definition. For $\beta \in B$,
\[ u_T (\mathfrak{b}(\beta) , \beta^{-1} \mathcal{D}_{\mathfrak{ff}^\prime} ) = \epsilon (\mathfrak{b}(\beta ), \beta^{-1}\mathcal{D}_{\mathfrak{ff}^\prime}, \pi^\prime) (\pi ^\prime)^{\zeta_{R,T}(H_{\mathfrak{ff}^\prime}/F, \mathfrak{b}(\beta),0)} \multint_{\mathbb{O}^\prime} x \ d \nu_T ( \mathfrak{b}(\beta), \beta^{-1} \mathcal{D}_{\mathfrak{ff}^\prime} , x ) . \] 

It is clear from the definition of $B$ that Theorem \ref{normforu1} follows from the following proposition.
\begin{proposition} \label{normu1toshow}
Let $\beta \in B$. We have
\begin{multline*}
    u_T (\mathfrak{b}(\beta) , \beta^{-1} \mathcal{D}_{\mathfrak{ff}^\prime} ) \\ = \left( \prod_{\epsilon \in E_+(\mathfrak{f})} \epsilon^{\nu^{ B }_T( \mathfrak{b}(\beta), \epsilon \beta^{-1} \mathcal{D}_{\mathfrak{f}} \cap \pi^{-1} \beta^{-1} \mathcal{D}_{\mathfrak{f}} ,  \mathcal{O}_\mathfrak{p} )} \right) \pi^{ \nu^{ B }_T ( \mathfrak{b}(\beta),  \mathcal{D}_{\mathfrak{f}},  \mathcal{O}_\mathfrak{p} )} \multint_{\mathbb{O}} x \ d \nu^{B }_T ( \mathfrak{b}(\beta),  \beta^{-1} \mathcal{D}_{\mathfrak{f}} ,  x ) .
\end{multline*}
\end{proposition}

The proof of Theorem \ref{normu1toshow} is largely an exercise in explicit calculation. We begin by considering the multiplicative integral in $u_T (\mathfrak{b}(\beta) , \beta^{-1} \mathcal{D}_{\mathfrak{ff}^\prime} )$.

\begin{lemma} \label{normu1multint}
We have
\begin{multline*}
    \multint_{\mathbb{O}^\prime} x \ d \nu_T ( \mathfrak{b}(\beta), \beta^{-1} \mathcal{D}_{\mathfrak{ff}^\prime} , x ) \\ = \left( \prod_{i=1}^{\alpha -1} \pi^{i \nu_T( \mathfrak{b}(\beta), \mathcal{D}_{\mathfrak{ff}^\prime}, \pi^i \mathbb{O} )} \right) \left( \prod_{i=0}^{\alpha -1} \prod_{\epsilon \in E_+(\mathfrak{ff}^\prime)} \epsilon^{\nu^{\{\pi^{-i} \}}_T ( \mathfrak{b}(\beta), \epsilon \beta^{-1} \mathcal{D}_{\mathfrak{ff}^\prime} \cap \pi^{-i} \beta^{-1} \mathcal{D}_{\mathfrak{ff}^\prime} ,  \mathbb{O} )} \right) \\ \left( \prod_{\gamma \in E_+(\mathfrak{f})/E_+(\mathfrak{ff}^\prime )} \gamma^{\nu^{A}_T ( \mathfrak{b}(\beta), \gamma \beta^{-1} \mathcal{D}_{\mathfrak{f}} ,  \mathbb{O} )} \right) \multint_{\mathbb{O}} x \ d \nu^{B}_T ( \mathfrak{b}(\beta),  \beta^{-1} \mathcal{D}_{\mathfrak{f}} ,  x ) .
\end{multline*}
\end{lemma}

\begin{proof}
Since $\pi^\prime = \pi^\alpha$ and $\mathbb{O}^\prime = \mathcal{O}_\mathfrak{p} - \pi^\prime \mathcal{O}_\mathfrak{p}$ we have $\mathbb{O}^\prime = \bigcup_{i=0}^{\alpha -1} \pi^i \mathbb{O} $. Then 
\[  \multint_{\mathbb{O}^\prime} x \ d \nu_T ( \mathfrak{b}(\beta), \beta^{-1} \mathcal{D}_{\mathfrak{ff}^\prime} , x ) = \prod_{i=0}^{\alpha -1} \multint_{\pi^i \mathbb{O}} x \ d \nu_T ( \mathfrak{b}(\beta), \beta^{-1} \mathcal{D}_{\mathfrak{ff}^\prime} , x ). \]
By changing variables and then factoring out $\pi^i$ we have 
\begin{align*}
    I(\beta) & \coloneqq \left( \prod_{i=1}^{\alpha -1} \pi^{i \nu_T( \mathfrak{b}(\beta), \mathcal{D}_{\mathfrak{ff}^\prime}, \pi^i \mathbb{O} )} \right)  \prod_{i=0}^{\alpha -1} \multint_{\mathbb{O}} x \ d \nu_T ( \mathfrak{b}(\beta), \beta^{-1} \mathcal{D}_{\mathfrak{ff}^\prime} , \pi^i x ) \\
    &= \left( \prod_{i=1}^{\alpha -1} \pi^{i \nu_T( \mathfrak{b}(\beta), \mathcal{D}_{\mathfrak{ff}^\prime}, \pi^i \mathbb{O} )} \right)  \prod_{i=0}^{\alpha -1} \multint_{\mathbb{O}} x \ d \nu^{\{\pi^{-i} \}}_T ( \mathfrak{b}(\beta), \pi^{-i} \beta^{-1} \mathcal{D}_{\mathfrak{ff}^\prime} ,  x ) .
\end{align*}
We now note that we can write, for $i=1, \dots , \alpha -1$,
\[ \pi^{-i} \mathcal{D}_{\mathfrak{ff}^\prime} = \bigcup_{\epsilon \in E_+(\mathfrak{ff}^\prime)} ( \epsilon \mathcal{D}_{\mathfrak{ff}^\prime} \cap \pi^{-i} \mathcal{D}_{\mathfrak{ff}^\prime} ) . \]
Then,
\begin{align*}
    &\prod_{i=0}^{\alpha -1} \multint_{\mathbb{O}} x \ d \nu^{\{\pi^{-i} \}}_T ( \mathfrak{b}(\beta), \pi^{-i} \beta^{-1} \mathcal{D}_{\mathfrak{ff}^\prime} ,  x ) \\ = & \prod_{i=0}^{\alpha -1} \prod_{\epsilon \in E_+(\mathfrak{ff}^\prime)} \multint_{\mathbb{O}} x \ d \nu^{\{\pi^{-i} \}}_T ( \mathfrak{b}(\beta), \epsilon \beta^{-1} \mathcal{D}_{\mathfrak{ff}^\prime} \cap \pi^{-i} \beta^{-1} \mathcal{D}_{\mathfrak{ff}^\prime} ,  x ) \\ = &
    \left( \prod_{i=0}^{\alpha -1} \prod_{\epsilon \in E_+(\mathfrak{ff}^\prime)} \epsilon^{\nu^{\{\pi^{-i} \}}_T ( \mathfrak{b}(\beta), \epsilon \beta^{-1} \mathcal{D}_{\mathfrak{ff}^\prime} \cap \pi^{-i} \beta^{-1} \mathcal{D}_{\mathfrak{ff}^\prime} ,  \mathbb{O} )} \right) \multint_{\mathbb{O}} x \ d \nu^{A}_T ( \mathfrak{b}(\beta),  \beta^{-1} \mathcal{D}_{\mathfrak{ff}^\prime} ,  x ).
\end{align*}
Here $A=\{1, \pi^{-1} , \dots , \pi^{\alpha -1} \}$. Then since $\mathcal{D}_{\mathfrak{ff}^\prime} = \bigcup_{\gamma \in E_+(\mathfrak{f})/E_+(\mathfrak{ff}^\prime )} \gamma \mathcal{D}_\mathfrak{f} $ we can write 
\[ \multint_{\mathbb{O}} x \ d \nu^{A}_T ( \mathfrak{b}(\beta),  \beta^{-1} \mathcal{D}_{\mathfrak{ff}^\prime} ,  x ) = \left( \prod_{\gamma \in E_+(\mathfrak{f})/E_+(\mathfrak{ff}^\prime )} \gamma^{\nu^{A}_T ( \mathfrak{b}(\beta), \gamma \beta^{-1} \mathcal{D}_{\mathfrak{f}} ,  \mathbb{O} )} \right) \multint_{\mathbb{O}} x \ d \nu^{\langle A, E \rangle }_T ( \mathfrak{b}(\beta),  \beta^{-1} \mathcal{D}_{\mathfrak{f}} ,  x )  \]
where $E= E_+(\mathfrak{f})/ E_+(\mathfrak{ff}^\prime)$. Thus we have, noting that $B= AE =\{ ae \mid a \in A, e \in E \}$,
\begin{multline*}
    I(\beta)  = \left( \prod_{i=1}^{\alpha -1} \pi^{i \nu_T( \mathfrak{b}(\beta), \mathcal{D}_{\mathfrak{ff}^\prime}, \pi^i \mathbb{O} )} \right) \left( \prod_{i=0}^{\alpha -1} \prod_{\epsilon \in E_+(\mathfrak{ff}^\prime)} \epsilon^{\nu^{\{\pi^{-i} \}}_T ( \mathfrak{b}(\beta), \epsilon \beta^{-1} \mathcal{D}_{\mathfrak{ff}^\prime} \cap \pi^{-i} \beta^{-1} \mathcal{D}_{\mathfrak{ff}^\prime} ,  \mathbb{O} )} \right) \\ \left( \prod_{\gamma \in E_+(\mathfrak{f})/E_+(\mathfrak{ff}^\prime )} \gamma^{\nu^{A}_T ( \mathfrak{b}(\beta), \gamma \beta^{-1} \mathcal{D}_{\mathfrak{f}} ,  \mathbb{O} )} \right) \multint_{\mathbb{O}} x \ d \nu^{B }_T ( \mathfrak{b}(\beta),  \beta^{-1} \mathcal{D}_{\mathfrak{f}} ,  x ) .
\end{multline*}
\end{proof}

We now consider the powers of $\pi$ given in the definition of $u_T (\mathfrak{b}(\beta) , \beta^{-1} \mathcal{D}_{\mathfrak{ff}^\prime} )$ and arising in the statement of Lemma \ref{normu1multint}. Recall that $\pi^\prime = \pi^\alpha$.

\begin{lemma} \label{normforu1lemma2}
\[ \left( \prod_{i=1}^{\alpha -1} \pi^{i \nu_T( \mathfrak{b}(\beta), \mathcal{D}_{\mathfrak{ff}^\prime}, \pi^i \mathbb{O} )} \right) \pi ^{\alpha \zeta_{R,T}(H_{\mathfrak{ff}^\prime}/F, \mathfrak{b}(\beta),0)} = \pi^{ \nu^{ B }_T ( \mathfrak{b}(\beta),  \mathcal{D}_{\mathfrak{f}},  \mathcal{O}_\mathfrak{p} )} . \]
\end{lemma}

\begin{proof}
Since $\pi^i \mathbb{O} = \pi^i\mathcal{O}_\mathfrak{p} - \pi^{i+1}\mathcal{O}_\mathfrak{p}$ we have by a telescope argument
\[ \sum_{i=1}^{\alpha -1} i \nu_T( \mathfrak{b}(\beta), \mathcal{D}_{\mathfrak{ff}^\prime}, \pi^i \mathbb{O} ) = -(\alpha -1) \nu_T( \mathfrak{b}(\beta), \mathcal{D}_{\mathfrak{ff}^\prime}, \pi^\alpha \mathcal{O}_\mathfrak{p} ) + \sum_{i=1}^{\alpha -1}  \nu_T( \mathfrak{b}(\beta), \mathcal{D}_{\mathfrak{ff}^\prime}, \pi^i \mathcal{O}_\mathfrak{p} ) . \]
Recalling the definition of $\mathcal{D}_{\mathfrak{ff}^\prime}$ we also note that for $i=0, \dots , \alpha -1$ we have
\[ \nu_T( \mathfrak{b}(\beta), \mathcal{D}_{\mathfrak{ff}^\prime}, \pi^i \mathcal{O}_\mathfrak{p} ) = \nu^E_T( \mathfrak{b}(\beta), \mathcal{D}_{\mathfrak{f}}, \pi^i \mathcal{O}_\mathfrak{p} ). \]
Thus we can calculate, using the fact that $\zeta_{R,T}(H_{\mathfrak{ff}^\prime}/F, \mathfrak{b}(\beta),0) = \nu_T( \mathfrak{b}(\beta), \mathcal{D}_{\mathfrak{ff}^\prime}, \mathcal{O}_\mathfrak{p} )$,
\begin{align*}
     & \left( \prod_{i=1}^{\alpha -1} \pi^{i \nu_T( \mathfrak{b}(\beta), \mathcal{D}_{\mathfrak{ff}^\prime}, \pi^i \mathbb{O} )} \right) \pi ^{\alpha \zeta_{R,T}(H_{\mathfrak{ff}^\prime}/F, \mathfrak{b}(\beta),0)} \\ =& \left( \prod_{i=1}^{\alpha -1} \pi^{ \nu_T( \mathfrak{b}(\beta), \mathcal{D}_{\mathfrak{ff}^\prime}, \pi^i \mathcal{O}_\mathfrak{p} )} \right) \pi ^{ -(\alpha -1) \nu_T( \mathfrak{b}(\beta), \mathcal{D}_{\mathfrak{ff}^\prime}, \pi^\alpha \mathcal{O}_\mathfrak{p} )}  \pi ^{\alpha \nu_T( \mathfrak{b}(\beta), \mathcal{D}_{\mathfrak{ff}^\prime}, \mathcal{O}_\mathfrak{p} )} \\
    =& \left( \prod_{i=1}^{\alpha -1} \pi^{ \nu^E_T ( \mathfrak{b}(\beta), \mathcal{D}_{\mathfrak{f}}, \pi^i \mathcal{O}_\mathfrak{p} )} \right) \pi ^{ -(\alpha -1) \nu^E_T ( \mathfrak{b}(\beta), \mathcal{D}_{\mathfrak{f}}, \pi^\alpha \mathcal{O}_\mathfrak{p} )}  \pi ^{\alpha \nu^E_T( \mathfrak{b}(\beta), \mathcal{D}_{\mathfrak{f}}, \mathcal{O}_\mathfrak{p} )} .
\end{align*}
By Lemma \ref{changeofvariable} we have for $i=1, \dots , \alpha $, 
\[ \nu^E_T ( \mathfrak{b}(\beta), \mathcal{D}_{\mathfrak{f}}, \pi^i \mathcal{O}_\mathfrak{p} ) = \nu^{E, \{ \pi^{-i} \} }_T ( \mathfrak{b}(\beta), \pi^{-i} \mathcal{D}_{\mathfrak{f}}, \mathcal{O}_\mathfrak{p} ) . \]
We can then write $\pi^{-i} \mathcal{D}_\mathfrak{f} = \bigcup_{\delta \in E_+(\mathfrak{f})} \delta \mathcal{D}_\mathfrak{f} \cap \pi^{-i} \mathcal{D}_\mathfrak{f} $. Then 
\begin{align*}
    \nu^{E, \{ \pi^{-i} \} }_T(\mathfrak{b}, \pi^{-i} \mathcal{D}_\mathfrak{f} , \mathcal{O}_\mathfrak{p}) &= \sum_{\delta \in E_+(\mathfrak{f})} \nu^{E, \{ \pi^{-i} \} }_T(\mathfrak{b}, \delta \mathcal{D}_\mathfrak{f} \cap \pi^{-i} \mathcal{D}_\mathfrak{f} , \mathcal{O}_\mathfrak{p}) \\
    &= \sum_{\delta \in E_+(\mathfrak{f})} \nu^{E, \{ \pi^{-i} \} }_T(\mathfrak{b},  \mathcal{D}_\mathfrak{f} \cap \delta \pi^{-i} \mathcal{D}_\mathfrak{f} , \mathcal{O}_\mathfrak{p}) \\
    &=\nu^{E, \{ \pi^{-i} \} }_T(\mathfrak{b},  \mathcal{D}_\mathfrak{f} , \mathcal{O}_\mathfrak{p}) .
\end{align*}
Remarking that $\pi^{-\alpha} = (\pi')^{-1} \equiv 1 \pmod{\mathfrak{f}\mathfrak{f}'}$, we deduce that
\[ \left( \prod_{i=1}^{\alpha -1} \pi^{i \nu_T( \mathfrak{b}(\beta), \mathcal{D}_{\mathfrak{ff}^\prime}, \pi^i \mathbb{O} )} \right) \pi ^{\alpha \zeta_{R,T}(H_{\mathfrak{ff}^\prime}/F, \mathfrak{b}(\beta),0)} = \prod_{i=0}^{\alpha -1} \pi^{ \nu^{E, \{ \pi^{-i} \} }_T ( \mathfrak{b}(\beta),  \mathcal{D}_{\mathfrak{f}},  \mathcal{O}_\mathfrak{p} )} = \pi^{ \nu^{ \langle A, E \rangle }_T ( \mathfrak{b}(\beta),  \mathcal{D}_{\mathfrak{f}},  \mathcal{O}_\mathfrak{p} )} . \]
Noting that $B=AE$ completes the proof.
\end{proof}

We now consider the error term in the definition of $u_T (\mathfrak{b}(\beta) , \beta^{-1} \mathcal{D}_{\mathfrak{ff}^\prime} )$ and the products of elements of $E_+(\mathfrak{f})$ that arise in Lemma \ref{normu1multint}. Considering Lemma \ref{normu1multint} and Lemma \ref{normforu1lemma2}, in order to prove Proposition \ref{normu1toshow} it is enough to prove the following.

\begin{proposition} \label{normu1lastprop}
Let
\begin{multline*}
    \Err(\beta)  = \epsilon (\mathfrak{b}(\beta ), \beta^{-1}\mathcal{D}_{\mathfrak{ff}^\prime}, \pi^\prime)  \left( \prod_{i=0}^{\alpha -1} \prod_{\epsilon \in E_+(\mathfrak{ff}^\prime)} \epsilon^{\nu^{\{\pi^{-i} \}}_T ( \mathfrak{b}(\beta), \epsilon \beta^{-1} \mathcal{D}_{\mathfrak{ff}^\prime} \cap \pi^{-i} \beta^{-1} \mathcal{D}_{\mathfrak{ff}^\prime} ,  \mathbb{O} )} \right) \\ \times \left( \prod_{\gamma \in E_+(\mathfrak{f})/E_+(\mathfrak{ff}^\prime )} \gamma^{\nu^{A}_T ( \mathfrak{b}(\beta), \gamma \beta^{-1} \mathcal{D}_{\mathfrak{f}} ,  \mathbb{O} )} \right) .
\end{multline*}
Then
\[ \Err(\beta) = \prod_{\epsilon \in E_+(\mathfrak{f})} \epsilon^{\nu^{ B }_T( \mathfrak{b}(\beta), \epsilon \beta^{-1} \mathcal{D}_{\mathfrak{f}} \cap \pi^{-1} \beta^{-1} \mathcal{D}_{\mathfrak{f}} ,  \mathcal{O}_\mathfrak{p} )} . \]
\end{proposition}

For clarity, we shall perform the calculations required for this proposition in a few lemmas. 

\begin{lemma} \label{normu1lemma}
We have
\[ \Err(\beta) = \left( \prod_{\epsilon \in E_+(\mathfrak{f})} \epsilon^{\nu^E_T( \mathfrak{b}(\beta), \epsilon \beta^{-1} \mathcal{D}_{\mathfrak{f}} \cap \pi^{-\alpha} \beta^{-1} \mathcal{D}_{\mathfrak{f}} ,  \mathcal{O}_\mathfrak{p} )} \right) \left( \prod_{i=1}^{\alpha-1}   \prod_{\epsilon \in E_+(\mathfrak{f})} \epsilon^{\nu^{ \langle E, \{\pi^{-i} \} \rangle }_T ( \mathfrak{b}(\beta), \epsilon \beta^{-1} \mathcal{D}_{\mathfrak{f}} \cap \pi^{-i} \beta^{-1} \mathcal{D}_{\mathfrak{f}} ,  \mathbb{O} )} \right) . \]
\end{lemma}

\begin{proof}
Considering the definition of $\mathcal{D}_{\mathfrak{ff}^\prime}$ we calculate
\begin{align}
    & \epsilon (\mathfrak{b}(\beta ), \beta^{-1}\mathcal{D}_{\mathfrak{ff}^\prime}, \pi^\prime) \\ =& \prod_{\epsilon \in E_+(\mathfrak{ff}^\prime)} \epsilon^{\nu_T( \mathfrak{b}(\beta), \epsilon \beta^{-1} \mathcal{D}_{\mathfrak{ff}^\prime} \cap \pi^{-\alpha} \beta^{-1} \mathcal{D}_{\mathfrak{ff}^\prime} ,  \mathcal{O}_\mathfrak{p} )} \\
    =& \prod_{\epsilon \in E_+(\mathfrak{ff}^\prime)} \prod_{\gamma \in E_+(\mathfrak{f})/E_+(\mathfrak{ff}^\prime)} \epsilon^{\nu_T( \mathfrak{b}(\beta), \epsilon \gamma \beta^{-1} \mathcal{D}_{\mathfrak{f}} \cap \pi^{-\alpha} \beta^{-1} \mathcal{D}_{\mathfrak{ff}^\prime} ,  \mathcal{O}_\mathfrak{p} )} \\
    =& \left( \prod_{\gamma \in E_+(\mathfrak{f})/E_+(\mathfrak{ff}^\prime)} \gamma^{-\nu_T( \mathfrak{b}(\beta), \gamma \beta^{-1} \mathcal{D}_{\mathfrak{f}}  ,  \mathcal{O}_\mathfrak{p} )} \right)  \left( \prod_{\epsilon \in E_+(\mathfrak{f})} \epsilon^{\nu_T( \mathfrak{b}(\beta), \epsilon \beta^{-1} \mathcal{D}_{\mathfrak{f}} \cap \pi^{-\alpha} \beta^{-1} \mathcal{D}_{\mathfrak{ff}^\prime} ,  \mathcal{O}_\mathfrak{p} )} \right) .
    \label{a1}
\end{align}
Similarly we have
\begin{multline}
    \prod_{\epsilon \in E_+(\mathfrak{f})} \epsilon^{\nu_T( \mathfrak{b}(\beta), \epsilon \beta^{-1} \mathcal{D}_{\mathfrak{f}} \cap \pi^{-\alpha} \beta^{-1} \mathcal{D}_{\mathfrak{ff}^\prime} ,  \mathcal{O}_\mathfrak{p} )} \\ = \left( \prod_{\gamma \in E_+(\mathfrak{f})/E_+(\mathfrak{ff}^\prime)} \gamma^{\nu_T( \mathfrak{b}(\beta), \gamma \pi^{-\alpha} \beta^{-1} \mathcal{D}_{\mathfrak{f}}  ,  \mathcal{O}_\mathfrak{p} )} \right)  \left( \prod_{\epsilon \in E_+(\mathfrak{f})} \epsilon^{\nu^E_T( \mathfrak{b}(\beta), \epsilon \beta^{-1} \mathcal{D}_{\mathfrak{f}} \cap \pi^{-\alpha} \beta^{-1} \mathcal{D}_{\mathfrak{f}} ,  \mathcal{O}_\mathfrak{p} )} \right) .
    \label{a2}
\end{multline}
We also calculate for $i=1, \dots , \alpha -1$
\begin{multline}
    \prod_{\epsilon \in E_+(\mathfrak{ff}^\prime)} \epsilon^{\nu^{\{\pi^{-i} \}}_T ( \mathfrak{b}(\beta), \epsilon \beta^{-1} \mathcal{D}_{\mathfrak{ff}^\prime} \cap \pi^{-i} \beta^{-1} \mathcal{D}_{\mathfrak{ff}^\prime} ,  \mathbb{O} )} \\ = \left( \prod_{\gamma \in E_+(\mathfrak{f})/E_+(\mathfrak{ff}^\prime)} \gamma^{-\nu^{ \{ \pi^{-i} \} }_T( \mathfrak{b}(\beta), \gamma \beta^{-1} \mathcal{D}_{\mathfrak{f}}  ,  \mathbb{O} )+\nu^{ \{ \pi^{-i} \} }_T ( \mathfrak{b}(\beta), \gamma \pi^{-i} \beta^{-1} \mathcal{D}_{\mathfrak{f}}  ,  \mathbb{O} )} \right) \\ \prod_{\epsilon \in E_+(\mathfrak{f})} \epsilon^{\nu^{\langle E , \{\pi^{-i} \} \rangle }_T ( \mathfrak{b}(\beta), \epsilon \beta^{-1} \mathcal{D}_{\mathfrak{f}} \cap \pi^{-i} \beta^{-1} \mathcal{D}_{\mathfrak{f}} ,  \mathbb{O} )} .
    \label{a3}
\end{multline}
We now note the following equalities, both of which hold via telescoping sum arguments.
\begin{enumerate}
    \item  \begin{multline*}
    \prod_{i=1}^{\alpha-1} \left( \prod_{\gamma \in E_+(\mathfrak{f})/E_+(\mathfrak{ff}^\prime)} \gamma^{-\nu^{ \{ \pi^{-i} \} }_T( \mathfrak{b}(\beta), \gamma \beta^{-1} \mathcal{D}_{\mathfrak{f}}  ,  \mathbb{O} )} \right) \left( \prod_{\gamma \in E_+(\mathfrak{f})/E_+(\mathfrak{ff}^\prime )} \gamma^{\nu^{A}_T ( \mathfrak{b}(\beta), \gamma \beta^{-1} \mathcal{D}_{\mathfrak{f}} ,  \mathbb{O} )} \right) \\
    =  \prod_{\gamma \in E_+(\mathfrak{f})/E_+(\mathfrak{ff}^\prime )} \gamma^{\nu_T ( \mathfrak{b}(\beta), \gamma \beta^{-1} \mathcal{D}_{\mathfrak{f}} ,  \mathbb{O} )} 
\end{multline*}
    \item
    \begin{align*}
         &\prod_{i=1}^{\alpha-1} \left( \prod_{\gamma \in E_+(\mathfrak{f})/E_+(\mathfrak{ff}^\prime)} \gamma^{\nu^{ \{ \pi^{-i} \} }_T ( \mathfrak{b}(\beta), \gamma \pi^{-i} \beta^{-1} \mathcal{D}_{\mathfrak{f}}  ,  \mathbb{O} )}    \right) \\ = & \prod_{\gamma \in E_+(\mathfrak{f})/E_+(\mathfrak{ff}^\prime)} \gamma^{\nu_T ( \mathfrak{b}(\beta), \gamma  \beta^{-1} \mathcal{D}_{\mathfrak{f}}  ,  \pi \mathcal{O}_\mathfrak{p} ) - \nu^{ \{ \pi^{-\alpha} \} }_T ( \mathfrak{b}(\beta), \gamma \pi^{-\alpha} \beta^{-1} \mathcal{D}_{\mathfrak{f}}  ,   \mathcal{O}_\mathfrak{p} ) } \\
         =& \prod_{\gamma \in E_+(\mathfrak{f})/E_+(\mathfrak{ff}^\prime)} \gamma^{\nu_T ( \mathfrak{b}(\beta), \gamma  \beta^{-1} \mathcal{D}_{\mathfrak{f}}  ,  \pi \mathcal{O}_\mathfrak{p} ) - \nu_T ( \mathfrak{b}(\beta), \gamma \pi^{-\alpha} \beta^{-1} \mathcal{D}_{\mathfrak{f}}  ,   \mathcal{O}_\mathfrak{p} ) }.
    \end{align*}
\end{enumerate}
Combining these two equalites with the calculations in (\ref{a1}), (\ref{a2}) and (\ref{a3}) gives the result.
\end{proof}

If $\alpha =1$ then Lemma \ref{normu1lemma} is equivalent to Proposition \ref{normu1lastprop} and thus we are finished in the case $\alpha =1$. From this point on we assume that $\alpha > 1$.

\begin{lemma} \label{normu1lemma3}
If $\alpha >1$ then 
\begin{multline*}
    \Err(\beta) \\ = \left( \prod_{\epsilon \in E_+(\mathfrak{f})} \epsilon^{\nu^{ B }_T( \mathfrak{b}(\beta), \epsilon \beta^{-1} \mathcal{D}_{\mathfrak{f}} \cap \pi^{-1} \beta^{-1} \mathcal{D}_{\mathfrak{f}} ,  \mathcal{O}_\mathfrak{p} )} \right)  \left( \prod_{\delta \in E_+(\mathfrak{f})} \delta^{\nu^E_T( \mathfrak{b}(\beta), \delta \beta^{-1} \pi^{-1} \mathcal{D}_{\mathfrak{f}} \cap \pi^{-\alpha} \beta^{-1} \mathcal{D}_{\mathfrak{f}} ,  \mathcal{O}_\mathfrak{p} )} \right) \\
    \prod_{i=1}^{\alpha-1}   \prod_{\epsilon \in E_+(\mathfrak{f})} \epsilon^{- \nu^{ \langle E, \{\pi^{-i} \} \rangle }_T ( \mathfrak{b}(\beta), \epsilon \beta^{-1} \mathcal{D}_{\mathfrak{f}} \cap \pi^{-1} \beta^{-1} \mathcal{D}_{\mathfrak{f}} , \pi \mathcal{O}_\mathfrak{p}  )} \\
    \prod_{i=2}^{\alpha-1} \prod_{\delta \in E_+(\mathfrak{f})} \delta^{\nu^{ \langle E, \{\pi^{-i} \} \rangle }_T ( \mathfrak{b}(\beta), \epsilon \beta^{-1} \pi^{-1} \mathcal{D}_{\mathfrak{f}} \cap \pi^{-i} \beta^{-1} \mathcal{D}_{\mathfrak{f}} , \mathbb{O}  )} .
\end{multline*}
\end{lemma}

\begin{proof}
For $i=2, \dots , \alpha$ we have
\[ \pi^i \mathcal{D}_\mathfrak{f} = \bigcup_{\delta \in E_+(\mathfrak{f})} \pi^{-1} \delta \mathcal{D}_\mathfrak{f} \cap \pi^{-i} \mathcal{D}_\mathfrak{f} . \]
Thus, applying this to the result of Lemma \ref{normu1lemma}, we have 
\begin{multline*}
    \Err(\beta)\\ = 
    \left( \prod_{\epsilon \in E_+(\mathfrak{f})} \epsilon^{\nu^E_T( \mathfrak{b}(\beta), \epsilon \beta^{-1} \mathcal{D}_{\mathfrak{f}} \cap \pi^{-1} \beta^{-1} \mathcal{D}_{\mathfrak{f}} ,  \mathcal{O}_\mathfrak{p} )} \right) \left( \prod_{\delta \in E_+(\mathfrak{f})} \delta^{\nu^E_T( \mathfrak{b}(\beta), \delta \beta^{-1} \pi^{-1} \mathcal{D}_{\mathfrak{f}} \cap \pi^{-\alpha} \beta^{-1} \mathcal{D}_{\mathfrak{f}} ,  \mathcal{O}_\mathfrak{p} )} \right) \\
    \prod_{i=1}^{\alpha-1} \left(  \prod_{\epsilon \in E_+(\mathfrak{f})} \epsilon^{\nu^{ \langle E, \{\pi^{-i} \} \rangle }_T ( \mathfrak{b}(\beta), \epsilon \beta^{-1} \mathcal{D}_{\mathfrak{f}} \cap \pi^{-1} \beta^{-1} \mathcal{D}_{\mathfrak{f}} ,  \mathbb{O} )} \right. \\ \left. \prod_{\delta \in E_+(\mathfrak{f})} \delta^{\nu^{ \langle E, \{\pi^{-i} \} \rangle }_T ( \mathfrak{b}(\beta), \epsilon \beta^{-1} \pi^{-1} \mathcal{D}_{\mathfrak{f}} \cap \pi^{-i} \beta^{-1} \mathcal{D}_{\mathfrak{f}} ,  \mathbb{O} )} \right) .
\end{multline*}
Remarking that $\prod_{\delta \in E_+(\mathfrak{f})} \delta^{\nu^{ \langle E, \{\pi^{-1} \} \rangle }_T ( \mathfrak{b}(\beta), \epsilon \beta^{-1} \pi^{-1} \mathcal{D}_{\mathfrak{f}} \cap \pi^{-1} \beta^{-1} \mathcal{D}_{\mathfrak{f}} ,  \mathbb{O} )} =1 $, since $\epsilon \beta^{-1} \pi^{-1} \mathcal{D}_{\mathfrak{f}} \cap \pi^{-1} \beta^{-1} \mathcal{D}_{\mathfrak{f}} = \emptyset $, gives the result.
\end{proof}

If $\alpha =2$ it is straightforward to see that Lemma \ref{normu1lemma3} is equivalent to Proposition \ref{normu1lastprop} and thus we are also finished in the case $\alpha =2$. From this point on we assume that $\alpha > 2$. From Lemma \ref{normu1lemma3} one can see that to prove Proposition \ref{normu1lastprop} it is enough for us to show 
\begin{multline}
    1= \prod_{\delta \in E_+(\mathfrak{f})} \delta^{\nu^E_T( \mathfrak{b}(\beta), \delta \beta^{-1} \pi^{-1} \mathcal{D}_{\mathfrak{f}} \cap \pi^{-\alpha} \beta^{-1} \mathcal{D}_{\mathfrak{f}} ,  \mathcal{O}_\mathfrak{p} )}  \\
    \prod_{i=1}^{\alpha-1}   \prod_{\epsilon \in E_+(\mathfrak{f})} \epsilon^{- \nu^{ \langle E, \{\pi^{-i} \} \rangle }_T ( \mathfrak{b}(\beta), \epsilon \beta^{-1} \mathcal{D}_{\mathfrak{f}} \cap \pi^{-1} \beta^{-1} \mathcal{D}_{\mathfrak{f}} , \pi \mathcal{O}_\mathfrak{p}  )} \\
    \prod_{i=2}^{\alpha-1} \prod_{\delta \in E_+(\mathfrak{f})} \delta^{\nu^{ \langle E, \{\pi^{-i} \} \rangle }_T ( \mathfrak{b}(\beta), \epsilon \beta^{-1} \pi^{-1} \mathcal{D}_{\mathfrak{f}} \cap \pi^{-i} \beta^{-1} \mathcal{D}_{\mathfrak{f}} , \mathbb{O}  )} .
    \label{leftover}
\end{multline}
To do this we first show the following lemma.
\begin{lemma} \label{normu1induction}
We have that for $j=1, \dots , \alpha -1$ the right hand side of (\ref{leftover}) is equal to
\begin{multline*}
    e(j)=\left( \prod_{\delta \in E_+(\mathfrak{f})} \delta^{\nu^E_T( \mathfrak{b}(\beta), \delta \beta^{-1} \pi^{-j} \mathcal{D}_{\mathfrak{f}} \cap \pi^{-\alpha} \beta^{-1} \mathcal{D}_{\mathfrak{f}} ,  \mathcal{O}_\mathfrak{p} )} \right) \\
    \prod_{i=j}^{\alpha-1}   \prod_{\epsilon \in E_+(\mathfrak{f})} \epsilon^{- \nu^{ \langle E, \{\pi^{-i} \} \rangle }_T ( \mathfrak{b}(\beta), \epsilon \beta^{-1} \pi^{-(j-1)} \mathcal{D}_{\mathfrak{f}} \cap \pi^{-j} \beta^{-1} \mathcal{D}_{\mathfrak{f}} , \pi \mathcal{O}_\mathfrak{p}  )} \\
    \prod_{i=j+1}^{\alpha-1} \prod_{\delta \in E_+(\mathfrak{f})} \delta^{\nu^{ \langle E, \{\pi^{-i} \} \rangle }_T ( \mathfrak{b}(\beta), \delta \beta^{-1} \pi^{-j} \mathcal{D}_{\mathfrak{f}} \cap \pi^{-i} \beta^{-1} \mathcal{D}_{\mathfrak{f}} , \mathbb{O}  )} .
\end{multline*}
Note that for $j=\alpha -1$ the last product is empty. We also remark that it is implicit in the statement of this lemma that $e(1)=\dots =e(\alpha-1)$.
\end{lemma}
\begin{proof}
We prove this by induction. The case $j=1$ holds trivially. We now assume it holds for $j$ and prove the result for $j+1$, i.e., we show $e(j)=e(j+1)$. To do this we note that for $i=j+2, \dots , \alpha$, we have
\[ \pi^{-i} \mathcal{D}_\mathfrak{f} = \bigcup_{\kappa \in E_+(\mathfrak{f})} \pi^{-(j+1)} \kappa \mathcal{D}_\mathfrak{f} \cap \pi^{-i} \mathcal{D}_\mathfrak{f}. \]
Thus, $e(j)$ is equal to the product of the following elements:
\begin{align}
    \label{line1} &\left( \prod_{\delta \in E_+(\mathfrak{f})} \delta^{\nu^E_T( \mathfrak{b}(\beta), \delta \beta^{-1} \pi^{-j} \mathcal{D}_{\mathfrak{f}} \cap \pi^{-(j+1)} \beta^{-1} \mathcal{D}_{\mathfrak{f}} ,  \mathcal{O}_\mathfrak{p} )} \right) \left( \prod_{\kappa \in E_+(\mathfrak{f})} \kappa^{\nu^E_T( \mathfrak{b}(\beta), \kappa \beta^{-1} \pi^{-(j+1)} \mathcal{D}_{\mathfrak{f}} \cap \pi^{-\alpha} \beta^{-1} \mathcal{D}_{\mathfrak{f}} ,  \mathcal{O}_\mathfrak{p} )} \right) \\
    \label{line2} & \prod_{i=j}^{\alpha-1}   \prod_{\epsilon \in E_+(\mathfrak{f})} \epsilon^{- \nu^{ \langle E, \{\pi^{-i} \} \rangle }_T ( \mathfrak{b}(\beta), \epsilon \beta^{-1} \pi^{-(j-1)} \mathcal{D}_{\mathfrak{f}} \cap \pi^{-j} \beta^{-1} \mathcal{D}_{\mathfrak{f}} , \pi \mathcal{O}_\mathfrak{p}  )}  \\
    \label{line3} & \prod_{i=j+1}^{\alpha-1} \prod_{\delta \in E_+(\mathfrak{f})} \delta^{\nu^{ \langle E, \{\pi^{-i} \} \rangle }_T ( \mathfrak{b}(\beta), \delta \beta^{-1} \pi^{-j} \mathcal{D}_{\mathfrak{f}} \cap \pi^{-(j+1)} \beta^{-1} \mathcal{D}_{\mathfrak{f}} , \mathbb{O}  )}  \\
    \label{line4} &\prod_{i=j+2}^{\alpha-1} \prod_{\kappa \in E_+(\mathfrak{f})} \kappa^{\nu^{ \langle E, \{\pi^{-i} \} \rangle }_T ( \mathfrak{b}(\beta), \kappa \beta^{-1} \pi^{-(j+1)} \mathcal{D}_{\mathfrak{f}} \cap \pi^{-i} \beta^{-1} \mathcal{D}_{\mathfrak{f}} , \mathbb{O}  )} .
\end{align}
We remark that the first bracketed term in (\ref{line1}), and (\ref{line4}) are already products in $e(j+1)$. We now consider (\ref{line3}) and calculate that it is equal to 
\begin{multline}\label{newline3}
    \left( \prod_{i=j+1}^{\alpha-1} \prod_{\delta \in E_+(\mathfrak{f})} \delta^{\nu^{ \langle E, \{\pi^{-(i-1)} \} \rangle }_T ( \mathfrak{b}(\beta), \delta \beta^{-1} \pi^{-(j-1)} \mathcal{D}_{\mathfrak{f}} \cap \pi^{-j} \beta^{-1} \mathcal{D}_{\mathfrak{f}} , \pi \mathcal{O}_\mathfrak{p}   ) } \right) \\
    \left( \prod_{i=j+1}^{\alpha-1} \prod_{\delta \in E_+(\mathfrak{f})} \delta^{ - \nu^{ \langle E, \{\pi^{-i} \} \rangle }_T ( \mathfrak{b}(\beta), \delta \beta^{-1} \pi^{-j} \mathcal{D}_{\mathfrak{f}} \cap \pi^{-(j+1)} \beta^{-1} \mathcal{D}_{\mathfrak{f}} , \pi \mathcal{O}_\mathfrak{p}   )} \right) .
\end{multline}
We now consider the way the terms in (\ref{newline3}) interact with (\ref{line2}). Multiplying (\ref{newline3}) by (\ref{line2}) gives
\begin{multline}\label{line2and3}
    \left(\prod_{i=j+1}^{\alpha-1}   \prod_{\epsilon \in E_+(\mathfrak{f})} \epsilon^{- \nu^{ \langle E, \{\pi^{-i} \} \rangle }_T ( \mathfrak{b}(\beta), \epsilon \beta^{-1} \pi^{-j} \mathcal{D}_{\mathfrak{f}} \cap \pi^{-(j+1)} \beta^{-1} \mathcal{D}_{\mathfrak{f}} , \pi \mathcal{O}_\mathfrak{p}  )} \right) \\
    \left( \prod_{\epsilon \in E_+(\mathfrak{f})} \epsilon^{- \nu^{ \langle E, \{\pi^{-(\alpha-1)} \} \rangle }_T ( \mathfrak{b}(\beta), \epsilon \beta^{-1} \pi^{-(j-1)} \mathcal{D}_{\mathfrak{f}} \cap \pi^{-j} \beta^{-1} \mathcal{D}_{\mathfrak{f}} , \pi \mathcal{O}_\mathfrak{p}  )} \right) .
\end{multline}
The first term in (\ref{line2and3}) is the term we were missing from $e(j+1)$. Thus it only remains to show that the second bracketed term in (\ref{line1}) multiplied by the second bracketed term in (\ref{line2and3}) is equal to $1$. This is shown by the following calculation,
\begin{align*}
    & \prod_{\delta \in E_+(\mathfrak{f})} \delta^{\nu^E_T( \mathfrak{b}(\beta), \delta \beta^{-1} \pi^{-j} \mathcal{D}_{\mathfrak{f}} \cap \pi^{-(j+1)} \beta^{-1} \mathcal{D}_{\mathfrak{f}} ,  \mathcal{O}_\mathfrak{p} )} \\
    &= \prod_{\delta \in E_+(\mathfrak{f})} \delta^{\nu^{ \langle  E, \{ \pi \} \rangle }_T( \mathfrak{b}(\beta), \delta \beta^{-1} \pi^{-(j-1)} \mathcal{D}_{\mathfrak{f}} \cap \pi^{-j} \beta^{-1} \mathcal{D}_{\mathfrak{f}} , \pi  \mathcal{O}_\mathfrak{p} )} \\
    &= \prod_{\delta \in E_+(\mathfrak{f})} \delta^{\nu^{ \langle  E, \{ \pi^{-(\alpha -1)} \} \rangle }_T( \mathfrak{b}(\beta), \delta \beta^{-1} \pi^{-(j-1)} \mathcal{D}_{\mathfrak{f}} \cap \pi^{-j} \beta^{-1} \mathcal{D}_{\mathfrak{f}} , \pi  \mathcal{O}_\mathfrak{p} )}.
\end{align*}
We therefore deduce that 
\[ e(j)=e(j+1) \]
as claimed. This completes the proof of the lemma.
\end{proof}

We are now ready to prove Proposition \ref{normu1lastprop}.
\begin{proof}[Proof of Proposition \ref{normu1lastprop}]
We consider $e(\alpha-1)$. From Lemma \ref{normu1induction}, we have that $e(\alpha -1)$ is equal to the right hand side of (\ref{leftover}). Then
\begin{multline*}
    e(\alpha-1)=\left( \prod_{\delta \in E_+(\mathfrak{f})} \delta^{\nu^E_T( \mathfrak{b}(\beta), \delta \beta^{-1} \pi^{-(\alpha-1)} \mathcal{D}_{\mathfrak{f}} \cap \pi^{-\alpha} \beta^{-1} \mathcal{D}_{\mathfrak{f}} ,  \mathcal{O}_\mathfrak{p} )} \right) \\
    \prod_{\epsilon \in E_+(\mathfrak{f})} \epsilon^{- \nu^{ \langle E, \{\pi^{-(\alpha-1)} \} \rangle }_T ( \mathfrak{b}(\beta), \epsilon \beta^{-1} \pi^{-(\alpha-2)} \mathcal{D}_{\mathfrak{f}} \cap \pi^{-(\alpha -1)} \beta^{-1} \mathcal{D}_{\mathfrak{f}} , \pi \mathcal{O}_\mathfrak{p}  )} .
\end{multline*}
Since $ \pi \equiv \pi^{-(\alpha-1)} \pmod{\mathfrak{f}\mathfrak{f}'}$, it is clear that
\[ e(\alpha-1) =1 . \]
This completes the proof of Proposition \ref{normu1lastprop} and thus proves Theorem \ref{normforu1}.
\end{proof}

\subsection{Norm compatibility for $u_3$}

We prove norm compatibility for $u_2$ and then obtain the result for $u_3$ by Theorem \ref{u2equ3}.

We recall the definition
\[ u_2= \sum_{\sigma \in G} u_{2}(\sigma) \otimes [\sigma^{-1}]=  \Eis_F^0 \cap \Delta_\ast(c_\id \cap \rho_{H/F}). \]
\begin{theorem} \label{normcompatforu2}
We have for any $\sigma \in G$,
\[ u_2(\sigma , H)  = \prod_{\substack{\tau \in G^\prime \\ \tau \mid_H = \sigma } } u_2(\tau , H^\prime )  . \]
\end{theorem}

\begin{remark}
This theorem has been stated without proof by the first and third authors in \cite[Proposition 6.3]{MR3861805}. We include the proof for completeness. We note also that the proof of the norm compatibility for $u_2$ is much simpler than that for $u_1$. This is a result of the additional structure we have due to the cohomological nature of the construction.
\end{remark}

\begin{proof}[Proof of Theorem \ref{normcompatforu2}]
We consider the natural map
\begin{align*}
    \psi : F_\mathfrak{p}^\ast \otimes \mathbb{Z}[G^\prime] & \rightarrow F_\mathfrak{p}^\ast \otimes \mathbb{Z}[G] \\
    \sum_{\tau \in G^\prime} n_\tau \otimes [\tau] & \mapsto \sum_{\sigma \in G} (\prod_{\substack{\tau \in G^\prime \\ \tau \mid_H = \sigma } } n_\tau ) \otimes [\sigma ] .
\end{align*}
Then, on the one hand,
\[ \psi( u_2(H^\prime) ) = \sum_{\sigma \in G} ( \prod_{\substack{\tau \in G^\prime \\ \tau \mid_H = \sigma } } u_2(\tau , H^\prime)  ) \otimes [ \sigma ] . \]
On the other hand
\begin{align*}
    \psi( u_2(H^\prime) ) &= \psi ( \Eis_F^0 \cap \Delta_\ast(c_\id \cap \rho_{H^\prime/F})  ) \\
    &= \Eis_F^0 \cap \psi_\ast \Delta_\ast(c_\id \cap \rho_{H^\prime/F}) \\
    &= \Eis_F^0 \cap \Delta_\ast(c_\id \cap \psi_\ast \rho_{H^\prime/F}) .
\end{align*}
The only equality of note here is the final one. This follows since we can commute $\psi_\ast$ with $\Delta_\ast$, which is a consequence of the definitions in \S\ref{s:delta}. Then since $\psi_\ast \rho_{H^\prime/F} = \rho_{H / F} $, the desired result follows.
\end{proof}

\addcontentsline{toc}{section}{References}

\bibliography{bib}
\bibliographystyle{plain}

\end{document}